\documentclass[11pt,a4paper,final]{amsart}
\usepackage[utf8]{inputenc}
\usepackage[T1]{fontenc}
\usepackage[UKenglish]{babel}
\usepackage[a4paper,margin=1in]{geometry}

\usepackage{amsmath,amsthm,amsfonts,amssymb}
\usepackage{mathrsfs,dsfont}
\usepackage{graphicx}
\usepackage{url}
\usepackage{verbatim}
\usepackage{xcolor}

\usepackage[autostyle]{csquotes} 

\makeatletter
\def\namedlabel#1#2{\begingroup
    #2%
    \def\@currentlabel{#2}%
    \label{#1}\endgroup
}
\makeatother

\theoremstyle{plain}
\newtheorem{theorem}{Theorem}[section]
\newtheorem{corollary}[theorem]{Corollary}
\newtheorem{lemma}[theorem]{Lemma}
\newtheorem{proposition}[theorem]{Proposition}

\theoremstyle{definition}
\newtheorem{remark}[theorem]{Remark}
\newtheorem{example}[theorem]{Example}
\newtheorem{definition}[theorem]{Definition}

\newtheorem*{notation}{Notation}
\newtheorem*{assumptions}{Assumptions}

\numberwithin{equation}{section}

\renewcommand\labelenumi{\textup{\alph{enumi})}}

\renewcommand\theenumi\labelenumi
\makeatletter\renewcommand{\p@enumii}{}\makeatother 

\renewcommand{\leq}{\leqslant}

\renewcommand{\geq}{\geqslant}

\renewcommand{\Re}{\mathrm{Re}}

\newcommand{\loc}{\mathrm{loc}}

\DeclareMathOperator{\supp}{supp}

\DeclareMathOperator{\ess}{ess\,sup}
\DeclareMathOperator{\essi}{\mathop{ess\,inf}}

\DeclareMathOperator{\Spec}{Spec}

\newcommand{\cB}{\mathcal{B}}

\newcommand{\cM}{\mathcal{M}}

\newcommand{\real}{\mathds{R}}
\newcommand{\R}{\mathds{R}}
\newcommand{\Pp}{\mathds{P}}
\newcommand{\Ee}{\mathds{E}}
\newcommand{\I}{\mathds{1}}

\newcommand{\N}{\mathds{N}}
\newcommand{\pr}{\mathds{P}}

\newcommand{\ex}{\mathds{E}}

\newcommand{\nat}{\mathds{N}}
\newcommand{\re}{\mathop{\mathrm{Re}}}

\newcommand{\cut}{\mathcal{N}}

\begin{document}
\title[Quasi-ergodicity of compact semigroups]{Quasi-ergodicity of \linebreak compact strong Feller semigroups on $L^2$}

\author[K.~Kaleta]{Kamil Kaleta}
\address[K.~Kaleta]{Wroc\l aw University of Science and Technology, Faculty of Pure and Applied Mathematics,
Wyb. Wyspia\'nskiego 27, 50-370 Wroc\l aw, Poland. E-Mail: \textnormal{kamil.kaleta@pwr.edu.pl}}

\author[R.L.~Schilling]{Ren\'e L.\ Schilling}
\address[R.L.~Schilling]{TU Dresden, Fakult\"at Mathematik, Institut f\"{u}r Mathematische Stochastik, 01062 Dresden, Germany. E-Mail: \textnormal{rene.schilling@tu-dresden.de}}


\thanks{
K.~Kaleta was supported by the Alexander von Humboldt Foundation (Germany) and by the National Science Centre (Poland) project OPUS-18 2019/35/B/ST1/02421.
R.L.~Schilling was supported through the DFG-NCN Beethoven Classic 3 project SCHI419/11-1 \& NCN 2018/31/G/ST1/02252. We are grateful for stimulating discussions with David Berger (Dresden), Krzysztof Bogdan (Wroc{\l}aw), Tomasz Klimsiak (Toru\'n) and Pawe{\l} Sztonyk (Wroc{\l}aw) on the topics of this paper.
}

\begin{abstract}
    We study the quasi-ergodicity of compact strong Feller semigroups $U_t$, $t > 0$, on $L^2(M,\mu)$;  we assume that $M$ is a locally compact Polish space equipped with a locally finite Borel measue $\mu$. The operators $U_t$ are ultracontractive and positivity preserving, but not necessarily self-adjoint or normal. We are mainly interested in those cases where the measure $\mu$ is infinite and the semigroup is not intrinsically ultracontractive. We relate quasi-ergodicity on $L^p(M,\mu)$ and uniqueness of the quasi-stationary measure with the finiteness of the heat content of the semigroup (for large values of $t$) and with the progressive uniform ground state domination property. The latter property is equivalent to a variant of quasi-ergodicity which progressively propagates in space as $t \uparrow \infty$; the propagation rate is determined by the decay of $U_t \I_M(x)$. We discuss several applications and illustrate our results with examples. This includes a complete description of quasi-ergodicity for a large class of semigroups corresponding to non-local Schr\"odinger operators with confining potentials.
\end{abstract}

\maketitle

\section{Introduction}

Let $(M,d,\mu)$ be a metric measure space. The central theme of our contribution is the question whether a semigroup of compact operators $\left\{U_t: t \geq 0 \right\}$ acting on the scale of Lebesgue spaces $L^p(\mu) = L^p(M,\mu)$ admits some \enquote{ergodic} measure. More precisely, we look at expressions of the from
\begin{gather}\label{cond}
    \lim_{t\to\infty} \left| \frac{\sigma(U_t f)}{\sigma(U_t\I_M)} - m(f) \right| = 0
    \quad\text{where}\quad
    m(f) := \int f\,dm,\quad
    \sigma(g) := \int g\,d\sigma,
\end{gather}
and we establish conditions on $\left\{U_t: t \geq 0 \right\}$ such that the limit in \eqref{cond} i) exists,  ii) exists uniformly for certain families of measures $\sigma$ and functions $f$, iii) converges with a certain rate (in time and space), and iv) that $m$ is the unique quasi-stationary measure.  In many situations the rate can be explicitly given, and it is even exponential in time $t$. The measure $m$, which is sometimes called a \emph{quasi-ergodic}, has a very concrete form
\begin{gather*}
    m(E) = \frac{\int_E \psi_0\,d\mu}{\int_M \psi_0\,d\mu}
\end{gather*}
where $\psi_0$ is the eigenfunction of the $L^2(\mu)$-adjoint semigroup $\left\{U_t^*: t \geq 0 \right\}$ for the eigenvalue $e^{-\lambda_0 t}$. In the context of Schr\"{o}dinger operators and their semigroups, $\lambda_0$ is the ground state eigenvalue and $\psi_0$ is the corresponding eigenfunction (for the adjoint). If $\psi_0 \in L^1(\mu)$,  the measure $m$ is a finite measure, and it is a quasi-stationary (or quasi-invariant) measure of the semigroup $\left\{U_t: t \geq 0 \right\}$.

Since we do not assume that $U_t\I_M \equiv \I_M$ -- i.e., we do not assume conservativeness -- we can expect only some kind of quasi-ergodic behaviour. This is reflected in the structure of \eqref{cond} where we normalize $\sigma(U_t f)$ by the total mass $\sigma(U_t \I_M)$; if $U_t$ is the transition semigroup of some non-conservative stochastic process with initial distribution $\sigma$, this can be seen as conditioning on survival, i.e.\ $\sigma(U_t f)/\sigma(U_t\I_M) = \int \Ee^x\left[f(X_t) \mid \zeta > t)\right]\,\sigma(dx)$, $\zeta = \inf\{s > 0 : X_s\notin M\}$ being the life-time.

Large time evolution phenomena occur in various areas of mathematics and science, e.g.\ in PDEs, dynamical systems, stochastic processes, statistical and quantum physics, and they are often modeled by semigroups of operators. These applications require quite different state spaces, ranging from  discrete spaces (lattices, graphs, \dots) to manifolds, fractals and, of course, classical Euclidean spaces. In order to cover these situations, we work in a rather general setting, on a locally compact Polish space $N$ and $\mu$ is a locally finite Borel measure on the Borel sets $\mathcal{B}(M)$, and we consider sufficiently regular (but not necessarily self-adjoint or normal) operator semigroups acting (initially) on $L^2(M,\mu)$  This includes the most interesting case when $M$ is unbounded and/or $\mu$ is an infinite measure. We consider semigroups of compact integral operators which are positivity improving and ultracontractive, and have the strong Feller property -- see \eqref{A0}--\eqref{A3} and \eqref{A4} in Section \ref{sec:cpt} for a precise statement. This framework accomodates many situations which are typically discussed in the literature.

We focus on Schr\"odinger semigroups -- with confining potentials, based both on local (second order differential operators) \cite{bib:DS,bib:D,bib:S,bib:S3,bib:Si,bib:MPR,bib:MS} and non-local kinetic term operators (L\'evy operators, fractional and relativistic Laplacians, etc.) \cite{bib:CMS,bib:JW,bib:KS,bib:KwM,bib:CKW,bib:KL15a,bib:KSchi20} -- and on Feynman--Kac semigroups of Markov processes whose state spaces are general locally compact Polish spaces; this includes processes on fractals \cite{bib:Fuk,bib:Kus} and Markov chains on graphs and discrete spaces \cite{bib:Bar,bib:CKS,bib:MS-C1,bib:MS-C2}. Some specific cases and various examples are discussed in Section \ref{sec:examples}. There are many other examples that fit into our framework, but which we do not discuss here in detail, e.g.\ evolution semigroups of non-homogeneous second order differential operators (generating diffusion processes) and non-local L\'evy-type operators (generating jump L\'evy-type processes) with confining coefficients \cite{bib:BGL,bib:O,bib:OR,bib:W,bib:SW} or Markov processes killed upon exiting an unbounded domain \cite{bib:BB,bib:Kw,bib:CKW,bib:CKW2}.

We are aware of only a few papers that work in a similar direction, Takeda \cite{bib:T2}, Knobloch \& Partzsch \cite{bib:KP} and Zhang, Li \& Song \cite{bib:ZSS}. Takeda uses a Dirichlet form approach, his processes and semigroups are symmetric and his main objective is the existence and uniqueness of the quasi-invariant measure. The key technical assumption is a tightness condition on the resolvent (kernel) of the semigroup; this is, essentially, a compactness condition which is weaker than IUC (intrinsically ultracontractivity). The contributions by Knobloch \& Partzsch and Zhang \emph{et al.}\ are, in some sense, extremes: Knobloch \& Partzsch strive for a high level of generality in the basic setup, starting from a general Markov kernel assuming what they call \enquote{compact domination}. In the end they prove that intrinsic ultracontractivity (IUC) implies uniform quasi-ergodicity at an exponential rate. One should mention, that the only known examples satisfying the compact dominaton property are IUC kernels and semigroups. On the other hand, Zhang, Li \& Song consider Markov processes, they are topologically quite general, but assume that the underlying measure is finite.

 From the point of view of applications, both IUC and the finiteness of the underlying measure are rather restrictive conditions; for Schr\"odinger semigroups with confining potentials IUC requires very fast growth of the potential at infinity, see \cite{bib:DS,bib:B,bib:KL15a,bib:ChW2016}. For non-self-adjoint operators IUC was first studied by Kim \& Song \cite{bib:KSo3,bib:KSo4}.

Recently, we introduced in \cite{bib:KSchi20} the notion of \emph{progressive intrinsic ultracontractivity} in order to obtain sharp two-sided heat kernel estimates for non-IUC or not necessarily IUC semigroups, see also Definition~\ref{def:pGSD} and Remark~\ref{rem:opposite} further down. As it turns out, it is this notion that helps us to avoid IUC as well as the finiteness of the underlying base measure. Consequently, we are able to characterize the quasi-ergodic regularity for a fairly general class of compact semigroups and underlying state spaces, including  necessary and sufficient conditions for the existence and uniqueness of a quasi-stationary measure.

\medskip
Let us briefly summarize the main contributions of this paper. Recall that the heat content is defined as $Z(t) = \|U_t\I_M\|_{L^1(\mu)} = \|U_t^*\I_M\|_{L^1(\mu)}$.

\medskip
\paragraph{\bfseries Finite heat content and exponential quasi-ergodicity}
We begin with a fairly general, yet easily verified, sufficient condition on the semigroup $\left\{U_t: t\geq 0\right\}$ such that the convergence in \eqref{cond} is exponential in time, uniform for $f\in L^p$, $\|f\|_p\leq 1$, for any fixed $p \in [1,\infty]$, and it holds for any finite measure $\sigma$. In fact (cf.\ Theorem \ref{th:hcf_and_eqe}) if the heat content of the semigroup is finite for some $t_1>0$, then $\psi_0$ is integrable, the measure $m$ is the unique quasi-stationary measure of the semigroup and the first term appearing in \eqref{cond} is dominated by $c_1 e^{-c_2 t} \sigma(U_{t_1}\I_M)/\sigma(\phi_0) \left\|f\right\|_p$ -- here $t\gg 1$ is large, the constants $c_1, c_2>0$ do not depend on $\sigma$ or $f \in L^p$, and $\phi_0$ is the ground state of $\left\{U_t: t \geq 0 \right\}$. A similar reasoning applies to the adjoint semigroup $\left\{U_t^*: t \geq 0 \right\}$. Taking $\sigma(dy)= \delta_{x}(dy)$ as the Dirac-delta measure concentrated in $x \in M$, we obtain pointwise convergence with an explicit space-rate which is uniform on compact sets, see Corollary \ref{cor:hcf_uqe}. We will see in Section \ref{sec:F-K-general} that for a very general class of Feynman--Kac semigroups with potential $V(x)$ the heat content $Z(t_1)$ is finite if the function $\exp(-t_1 V(x))$ is integrable \enquote{at infinity}; this condition is often also a necessary condition. This means that the finiteness of the heat content is independent of the underlying Markov process. Since our proof relies only on the Feynman--Kac formula, see Demuth and van Casteren \cite{bib:DC}, it also applies to evolution semigroups of very general Schr\"odinger operators $H=-L+V$ -- in this case the finiteness of the heat content is determined solely by the behaviour of the potential $V$ at infinity, and it is independent of the kinetic term, i.e.\ the free Feller operator $L$.

\medskip
By definition, the heat content is finite if $U_t\I_M$ is integrable. The latter is guaranteed by the asymptotic ground state domination (aGSD) of the semigroup (that is $U_t\I_M \leq c e^{-t\lambda_0}\phi_0$ for some $t>0$) and the integrability of the ground state $\phi_0$. In \cite{bib:KSchi20} we defined the notion of progressive ground state domination (pGSD), which requires only that $\I_{K_t}U_t\I_M \leq c e^{-t\lambda_0}\phi_0$ for some increasing family of compact sets $K_t\uparrow M$. In \cite{bib:KSchi20} pGSD was obtained as a necessary condition of sharp two-sided heat kernel estimates for Schr\"{o}dinger semigroups, we will now see that pGSD is closely related to quasi-ergodicity.

\medskip
\paragraph{\bfseries Equivalence of aGSD/pGSD and quasi-ergodicity}
If pGSD holds for some family of compact sets $K_t\uparrow M$, one can show that \eqref{cond} is dominated by some rate $\kappa(t)$ which is essentially given by decay of $U_t\I_{M}(x)$ and $U_t^*\I_{M}(x)$ for $x\notin K_{bt}$ (for some fixed $b\in (0,1/2)$). In fact, the bound $\kappa(t)$ is uniform for all $f\in L^p(\mu)$, $\|f\|_p\leq 1$, and all measures $\sigma$ with support in $K_{at}$ (where $a\in (0,1)$ such that $a+2b=1$), see Theorem~\ref{th:pGSD-to-puqe}. Specializing again to $\sigma = \delta_x$ with $x\in K_{at}$, we see in Theorem~\ref{th:pGSD-pIUC} that pGSD is equivalent to the convergence of \eqref{cond} to zero (locally uniformly for $x\in K_{t}$) as $t\to\infty$. The latter property is best described as \enquote{progressive uniform quasi-ergodicity}. As a consequence, see Theorem~\ref{th:pGSD-gives-uniq}, we also see that the quasi-ergodic measure is necessarily given by $m$, resp., $m^*$, hence it is unique.

In Section~\ref{sec:pGSD_and_hc} we see that the finiteness of the heat content always implies pGSD, but without good control on the growth of the family of compact sets $K_t\uparrow M$. In this connection it is interesting to compare Theorem~\ref{th:hcf_and_eqe} and Theorems~\ref{th:pGSD-to-puqe} \& \ref{th:pGSD-pIUC}. Theorem~\ref{th:hcf_and_eqe} gives an easy-to-check sufficient condition (to wit:\ finite heat content) for the exponential (in time) quasi-ergodicity, but we may not be able to control the space behaviour. Theorem~\ref{th:pGSD-to-puqe} \& \ref{th:pGSD-pIUC} contain a necessary and sufficient condition (to wit:\ pGSD) for a somewhat stronger  property, which is progressively uniform in space, and with a time-rate depending on the semigroup. The examples in Section~\ref{subsec-schroedinger} illustrate that the finiteness of the heat content requires stronger assumptions on the generator -- for instance, for Schr\"odinger or Feynman--Kac semigroups one needs that the potential $V$ grows at infinity at least logarithmically.

\medskip
Assuming that the heat content is finite, our methods also yield results on the asymptotics of compact semigroups and the heat content as $t\to\infty$, cf.\ Section~\ref{sec-large-time}. For example, one can see that $Z(t)$ is exponentially close to $e^{-\lambda_0 t}\|\phi_0\|_1\|\psi_0\|_1/\mu(\psi_0\phi_0)$ as $t\to\infty$, see Corollary~\ref{cor:hc_asym}.

\medskip
The final section (Section~\ref{sec:examples}) contains a number of examples which illustrate the applicability and the limitations of our results from the previous sections. The examples in~\ref{sec:F-K-general} are related to Theorem \ref{th:hcf_and_eqe} and Corollary \ref{cor:hcf_uqe} and the examples in the other two Sections~\ref{subsec:HO}, \ref{subsec-schroedinger} mainly illustrate the pGSD-property from Theorems~\ref{th:pGSD-to-puqe} \& \ref{th:pGSD-pIUC}. In fact, Section \ref{subsec-schroedinger} contains a complete characterization, giving necessary and sufficient conditions, of quasi-ergodic regularity of the Schr\"odinger semigroups corresponding to a large class of non-local Schr\"odinger operators with confining potentials; this can be seen as a direct continuation of our recent work \cite{bib:KSchi20}.

\begin{notation}
Most of our notation is standard or self-explanatory. We use \enquote{positive} in the non-strict sense, e.g.\ a function
$f \in L^2(M,\mu)$ is called positive, if $f \geq 0$ $\mu$-a.e., and we say that $f$ is \enquote{strictly positive}, if $f > 0$ $\mu$-a.e. We will frequently use the shorthand $m(f)$ to denote $\int f\,dm$. We use $a\wedge b$ and $a\vee b$ to indicate $\min(a,b)$ and $\max(a,b)$, and we write $f\asymp g$ if $c^{-1} f(x)\leq g(x)\leq c g(x)$ for all $x$ (in a given range) and with some fixed constant $c\in (0,\infty)$. Finally, we call a strongly continuous contraction semigroup $\left\{U_t: t\geq 0\right\}$ a \emph{Feller semigroup} if it maps the continuous functions vanishing at infinity, $C_\infty(M)$ into itself, and a \emph{strong Feller semigroup} if it maps the bounded Borel measurable functions $B_b(M)$ (or, sometimes, the set $L^\infty(M)$) into the bounded continuous functions $C_b(M)$.
\end{notation}

\section{Setup and basic assumptions} \label{sec:cpt}

\paragraph{\bfseries Topological preliminaries.}
For the convenience of our readers we summarize a few topological preliminaries with references. We will work in a \emph{Polish space} $M$, i.e.\ a completely metrizable topological space such that the metric defines the topology, and the topology has a countable base (second countable). Being a metric space, $M$ is also first countable (each point has a countable neighbourhood basis) and it is separable, i.e.\ it contains a countable dense subset, cf.\ \cite[Corollary 1.3.8, Theorem 4.1.15]{bib:E}. We call $M$ \emph{locally compact} if each $x\in M$ has a relatively compact open neighbourhood $U(x)$. If a Polish space is locally compact, it is $\sigma$-compact (or countable at infinity), i.e.\ there is an increasing sequence of compact sets $K_n$ such that $\bigcup_{n\in\nat} K_n = M$, cf.\ \cite[Theorem 4.1.15.iii)]{bib:E}. In a metric space compact sets are closed and $\sigma$-compactness allows us to write any closed set as a countable union of compact sets. Therefore, the Borel $\sigma$-algebra $\mathcal B(M)$ of a locally compact Polish space is generated by the compact sets. A Borel measure $\mu$ is a positive measure defined on the Borel sets $\mathcal B(M)$. It is locally finite, if each point $x\in M$ has an open neighbourhood with finite measure. This implies that $\mu$ is finite on compact sets. Thus, a locally finite measure on a locally compact Polish space is $\sigma$-finite.

\medskip

\paragraph{\bfseries Setting and assumptions.}
Throughout we assume that  $M$ is a locally compact Polish space  and $\mu$ a positive, locally finite Borel measure with full topological support. We do not assume that the space $M$ is compact or that the measure $\mu$ is finite. On $L^2(M,\mu)$ we consider a strongly continuous semigroup of bounded linear operators, $\left\{U_t: t \geq 0 \right\}$ and its adjoint semigroup $\left\{U_t^*: t \geq 0 \right\}$, which is given by
\begin{gather*}
    \int U_t f(x) \cdot \overline{g(x)}\,\mu(dx) = \int f(x)\cdot \overline{U_t^* g(x)}\,\mu(dx), \quad f,g \in L^2(M,\mu),\; t>0.
\end{gather*}
Moreover, we assume that $U_t$ and $U_t^*$ are integral operators, which are given by a real-valued measurable kernel $u_t(x,y)$:
\begin{gather} \label{eq:kernel}
    U_t f(x) = \int u_t(x,y) f(y)\,\mu(dy), \qquad U_t^* g(y) = \int u_t(x,y) g(x) \,\mu(dx), \quad t>0.
\end{gather}
Moreover, we use the following assumptions:
\begin{assumptions}
\begin{align}
&\tag{\textbf{\upshape A0}}\label{A0}
\parbox[t]{.8\linewidth}{The linear operators $U_t : L^2(M,\mu)\to L^2(M,\mu)$, $t>0$, are compact;}
\\
&\tag{\textbf{\upshape A1}}\label{A1}
\parbox[t]{.8\linewidth}{For some $t_0 > 0$ the operator $U_{t_0}$ is positivity improving: for every positive $f \in L^2(M,\mu)$, $f\not\equiv 0$, we have $U_{t_0} f(x) > 0$ for $\mu$-almost all $x\in M$.}
\\
&\tag{\textbf{\upshape A2}}\label{A2}
\parbox[t]{.8\linewidth}{
The semigroups $\left\{U_t: t \geq 0 \right\}$ and $\left\{U_t^*: t \geq 0 \right\}$ have the strong Feller property: for every $t>0$ and $f \in L^{\infty}(M,\mu)$, we have $U_t f, U_{t}^*f \in C_b(M)$;}
\\
&\tag{\textbf{\upshape A3}}\label{A3}
\parbox[t]{.8\linewidth}{For some $t_0>0$ the operators $U_{t_0}, U_{t_0}^*: L^2(M,\mu) \to L^{\infty}(M,\mu)$ are bounded, i.e.\ ultracontractive:\\[10pt]
\mbox{}\hfill$\displaystyle
   \int u^2_{t_0}(\cdot,y) \,\mu(dy) \in L^{\infty}(M,\mu) \quad \text{and} \quad \int u^2_{t_0}(x,\cdot) \,\mu(dx) \in L^{\infty}(M,\mu).
$\hfill\mbox{}
}
\end{align}
\end{assumptions}
Our results in Sections~\ref{sec:pGSD_and_hc} and~\ref{sec:pGSD} need, when considering non-compact spaces $M$, a further \emph{strong decay} condition.
\begin{align}
\mbox{}\tag{\textbf{\upshape A4}}\label{A4}
&\parbox[t]{.8\linewidth}{There is some $t_0>0$ such that for every $t \geq t_0$ and $\epsilon >0$ there exists a compact set $K \subset M$ satisfying \\[5pt]
\mbox{}\hfill$\displaystyle
    U_t\I_M(x) \leq \epsilon \quad \text{or} \quad  U^*_t\I_M(x) \leq \epsilon, \quad\text{for every \; $x \in M \setminus K$.}
$\hfill\mbox{}}
\end{align}

Let us comment on the assumptions \eqref{A0}--\eqref{A4} and how they are related.
\begin{remark} \label{rem:discuss_ass}
\eqref{A0} is equivalent to assuming that the adjoint operators $U_t^*$ are compact on $(L^2(M,\mu))^* = L^2(M,\mu)$.

\eqref{A1} is equivalent to assuming that $U_{t_0}^*$ is positivity improving for some $t_0>0$. Moreover, $U_{t_0}$ and $U_{t_0}^*$ are positivity improving if, and only if, the integral kernel $u_{t_0}(x,y)$ is for $\mu\otimes\mu$ almost all $(x,y)$ strictly positive.

\eqref{A2} is equivalent to the condition that $U_t \I_M(x)$ and $U_t^* \I_M(x)$ are bounded continuous functions of $x$. These functions will play a central role in our investigations.

\eqref{A3} implies, because of the reflexivity of the space $L^2(M,\mu)$, that (the restrictions of the second duals to $L^1(M,\mu) \subset \big(L^{\infty}(M,\mu)\big)^*$ of) the operators $U_{t_0}$ and $U^*_{t_0}$ are bounded operators from $L^1(M,\mu)$ to $L^2(M,\mu)$. Due to the semigroup property, the operators $U_{2t_0}$ and $U^*_{2t_0}$ are bounded operators from $L^1(M,\mu)$ to $L^{\infty}(M,\mu)$; this is equivalent to saying
\begin{align} \label{eq:ultra}
    u_{2t_0}(\cdot,\cdot)  \in L^{\infty}(M \times M, \mu \otimes \mu).
\end{align}
Using the semigroup property, these boundedness properties extend to all $t \geq t_0$, resp.\ $t \geq 2 t_0$.
Conversely, if there is some $t_0>0$, for which \eqref{eq:ultra} holds, then the condition \eqref{A3} is also satisfied.

\medskip
    Let $t>2t_0$.
    From \eqref{A2} we know that $U_t$ and $U_t^*$ are continuous operators from $L^\infty(M,\mu)\to L^\infty(M,\mu)$, and \eqref{A3} means that $U_t$ and $U_t^*$ are continuous operators from $L^2(M,\mu)\to L^\infty(M,\mu)$; by duality, we see that $U_t$ and $U_t^*$ map $L^1(M,\mu)$ continuously into $L^2(M,\mu)$, and using again \eqref{A3} and the semigroup property, we conclude that $U_t$ and $U_t^*$ are continuous from $L^1(M,\mu)$ into $L^\infty(M,\mu)$. Therefore, a standard interpolation argument shows that $U_t$ and $U_t^*$ are continuous $L^p(M,\mu)\to L^\infty(M,\mu)$ for any $p\in (1,\infty)$ and, in view of \eqref{A2}, we also have
    \begin{align} \label{eq:Lp_cont}
        U_t \big(L^p(M,\mu)\big), U_{t}\big(L^p(M,\mu)\big) \subseteq C_b(M), \quad t>2t_0, \; p \in [1,\infty].
    \end{align}
\end{remark}

Let $A$, and $A^*$, denote the generators of semigroups $\left\{U_t: t \in T \right\}$, and $\left\{U_t^*: t \geq 0 \right\}$, respectively. A version of the Jentzsch theorem, see \cite[Theorem V.6.6]{bib:Sch}, shows that $\lambda_0:= \inf \re \big( \Spec (-A)\big) = \inf \re \big(\Spec (-A^*)\big)$ is a joint isolated and simple eigenvalue of $-A$ and $-A^*$. Moreover, the corresponding normalized eigenfunctions $\phi_0, \psi_0 \in L^2(M,\mu)$ can be assumed to be strictly positive. In particular,
\begin{align} \label{eq:eeqs}
    U_t \phi_0
    = e^{-\lambda_0 t} \phi_0
    \quad\text{and}\quad
    U_t^* \psi_0
    = e^{-\lambda_0 t} \psi_0, \quad t >0.
\end{align}
The number $\lambda_0$ is called the \emph{ground state eigenvalue}, and the functions $\phi_0, \psi_0$, are the \emph{ground states} of $-A$, and $-A^*$, respectively.

By \eqref{eq:Lp_cont} we have $\phi_0, \psi_0 \in C_b(M)$ and both eigenequations in \eqref{eq:eeqs} are also true in a pointwise sense. We will frequently use the following constant
\begin{gather*}
    \Lambda := \int \phi_0(x) \psi_0(x)\,\mu(dx)
\end{gather*}
and set
\begin{gather}\label{eq:cut}
    \cut := \left\{x \in M: \phi_0(x) = 0 \right\} \cup \left\{x \in M: \psi_0(x) = 0 \right\}.
\end{gather}
As mentioned above, $\mu(\cut)=0$. Due to the continuity of $\phi_0$ and $\psi_0$ the set $\cut$ is a closed set. If for some $t>0$ the kernel $u_t(x,y)$ is defined pointwise (e.g.\ if the map $(x,y) \mapsto u_t(x,y)$ is continuous) and strictly positive for all $x,y \in M$, then $\cut = \emptyset$. We are mainly interested in examples which come from the theory of stochastic processes, and for these examples $\cut = \emptyset$  is always satisfied. Since the original question comes from the $L^2$-theory, we do not want to assume \emph{a priori} $\cut = \emptyset$.

The following result is a straightforward adaptation of a Lemma by Zhang, Li \& Song \cite[Lemma 2.1]{bib:ZSS}; it is a variant of the estimate which was originally obtained for intrinsic semigroups by Kim \& Song \cite[Theorem 2.7]{bib:KSo2}. The very first result in this direction is due to Pinsky \cite[Theorem 3]{bib:Pin} for diffusion semigroups on bounded domains in $\R^d$. The result of Zhang \emph{et al.}\ is valid for rather general nonsymmetric semigroups. These authors use sub-Markovian semigroups of standard Markov processes with strictly positive kernels of Hilbert-Schmidt type satisfying \eqref{A3}. A close inspection of their proof reveals that only the compactness of the operators $U_t$ and $U^*_t$ is needed and that the result is valid in the more general setting, which is described by our present conditions \eqref{A0}--\eqref{A3}. Note that Zhang \emph{et al.}\ use $t_0=1$ and that their proof requires $t/2 > 1 = t_0$.

\begin{lemma}[{\cite[Lemma 2.1]{bib:ZSS}}]\label{lem:unif_conv}
    If \eqref{A0}--\eqref{A3} hold for some $t_0>0$,
    then there exist $C, \gamma > 0$ such that
    \begin{align}\label{eq:unif_conv}
        \left|e^{\lambda_0 t} u_t(x,y) - (1/\Lambda)\phi_0(x) \psi_0(y) \right|
        \leq C e^{-\gamma t}, \quad t > 2t_0, \quad\text{for $\mu\otimes\mu$ a.a.\ $x,y\in M$.}
    \end{align}
\end{lemma}
This estimate leads to a rate of convergence (as $t\to\infty$) for compact semigroups, which can be seen as a variant of the classical result of Chavel \& Karp~\cite{bib:ChK}, see also \cite{bib:S2,bib:KLVW}.

\medskip
\begin{center}\bfseries
    Unless otherwise stated, we take the same \boldmath $t_0>0$ \unboldmath in \eqref{A1}, \eqref{A3} and \eqref{A4}.
\end{center}
\bigskip

\noindent
We will now briefly discuss the relation between \eqref{A4} and \eqref{A0}.
\begin{lemma} \label{lem:a4_gives_a0}
    Assume \eqref{A2} and \eqref{A3}, \eqref{A4} for some $t_0>0$. Then the operators $U_t, U_t^*: L^2(M,\mu) \to L^2(M,\mu)$ are compact for all $t \geq t_0$. In particular, if \eqref{A3}, \eqref{A4} are true for every $t_0>0$, then \eqref{A0} holds.
\end{lemma}

\begin{proof}
We only consider the operators $U_t$. The proof for $U^*_t$ is similar. For a compact set $K \subset M$ we define:
\begin{gather*}
    U_t^K f(x) := \I_K(x) U_t f(x) = \int_{M} \I_K(x)u_t(x,y) f(y) \, \mu(dy), \quad f \in L^2(M,\mu), \ t>0,
\end{gather*}
i.e.\ $U_t^K$ is an integral operator with the kernel $u^K_t(x,y) = \I_K(x)u_t(x,y)$. Observe that by \eqref{A3},
\begin{align*}
    \int_M\int_M u^K_t(x,y)^2 \,\mu(dy)\,\mu(dx)
    & = \int_K\int_M u_t(x,y)^2 \,\mu(dy)\,\mu(dx) \\
    & \leq \mu(K) \cdot \ess_{x \in M} \int_M u_t(x,y)^2 \,\mu(dy)
    < \infty, \quad t \geq t_0,
\end{align*}
which means that $U^K_t$, $t \geq t_0$, are Hilbert--Schmidt operators. Moreover, for $f \in L^2(M,\mu)$,
by the Cauchy--Schwarz inequality and the Tonelli theorem,
\begin{align*}
    \int_M \left|U_t f(x) - U^K_t f(x) \right|^2 \,\mu(dx)
    & = \int_{M \setminus K} \left|U_t f(x)\right|^2 \,\mu(dx) \\
    & \leq \int_{M \setminus K} U_t \I_M(x) \int_M u_t(x,y) |f(y)|^2 \,\mu(dy) \,\mu(dx) \\
	& \leq \sup_{x \in {M \setminus K}} U_t \I_M(x) \left\|U^*_t \I_{M}\right\|_{\infty} \left\|f\right\|_2^2.
\end{align*}
The assumptions \eqref{A2} and \eqref{A4} yield that for every $t > t_0$ and $\epsilon >0$ we can find a compact set $K$ such that
\begin{gather*}
    \left\|U_t f - U^K_t f\right\|_2 \leq \epsilon \left\|f\right\|_2.
\end{gather*}
Letting $\epsilon\to 0$, we see that every operator $U_t$, $t \geq t_0$, is a strong limit of compact operators, hence compact. This completes the proof.
\end{proof}

If the operators $U_t$, $t>0$, are self-adjoint, then compactness for some $t >0$ implies compactness for every $t>0$ -- this is due to spectral theorem. In that case, \eqref{A4} implies \eqref{A0}, even if \eqref{A3}, \eqref{A4} are true for some $t_0>0$ only. For non-self-adjoint operators this is not clear. Therefore, in Sections~\ref{sec:pGSD_and_hc} and~\ref{sec:pGSD}, we have to assume both \eqref{A0} and \eqref{A4}.

In unbounded spaces, it is often true that \eqref{A0} and \eqref{A4} (holding for all $t_0>0$) are equivalent, see e.g.\ \cite[Lemma 1]{bib:Kw} and \cite[Lemma 9]{bib:KaKu}. Some general sufficient conditions for compactness in various settings, including Markov, Schr\"odinger and Feynman--Kac semigroups, and more general perturbations can be found in \cite{bib:WW,bib:T1,bib:LSW}.

\section{Heat content and exponential quasi-ergodicity of compact semigroups} \label{sec:content}

If $\phi_0, \psi_0 \in L^1(M,\mu)$,\footnote{In general, $\phi_0, \psi_0 \in L^1(M,\mu)$ need not be satisfied since $\mu$ can be an infinite measure.} we can define  probability measures on $(M,\cB(M))$ by
\begin{align} \label{eq:quasi_stat_meas}
    m(E) = \frac{\int_E \psi_0(x)\,\mu(dx)}{\int_M \psi_0(x)\,\mu(dx)}
    \quad\text{and}\quad
    m^*(E) = \frac{\int_E \phi_0(x)\,\mu(dx)}{\int_M \phi_0(x)\,\mu(dx)}.
\end{align}

It is rather easy to check that the measures $m$ and $m^*$ are \emph{quasi-stationary} (\emph{probability}) \emph{measures} of the semigroups $\left\{U_t: t \geq 0 \right\}$ and $\left\{U^*_t: t \geq 0 \right\}$.

\begin{definition} [\textit{Quasi-stationary measure on $L^{\infty}$}] \label{def:qsm}
    Let \eqref{A2} hold. A probability measure $\bar m$ on $(M,\cB(M))$ is said to be a \emph{quasi-stationary measure} of the semigroup $\left\{U_t: t \geq 0 \right\}$ (on $L^{\infty}(M,\mu)$), if for all $t>0$ and $f \in L^{\infty}(M,\mu)$ we have $\bar m(U_t\I_M) >0$ and
	\begin{gather}\label{eq:quasi-stat-def}
    \frac{\bar m(U_tf)}{\bar m(U_t\I_M)} = \bar m(f).
\end{gather}
\end{definition}

Because of \eqref{eq:Lp_cont} -- this follows from \eqref{A2}, \eqref{A3} -- we can extend the definition of quasi-stationarity to $L^p(M,\mu)$ for any $p \in [1,\infty)$, at least for large $t$.

\begin{lemma}\label{lem:qsm_Lp}
    Assume \eqref{A2}--\eqref{A3} for $t_0>0$ and let $p \in [1,\infty)$. Let $\bar m$ be a probability measure on $(M,\cB(M))$ and let $t > 2t_0$ be such that $\bar m(U_t\I_M) >0$. Then
\begin{gather*}
    \forall_{g \in L^p(M,\mu)} \quad \frac{\bar m(U_tg)}{\bar m(U_t\I_M)} = \bar m(g)
    \quad\iff\quad
    \forall_{f \in L^{\infty}(M,\mu)} \quad  \frac{\bar m(U_tf)}{\bar m(U_t\I_M)} = \bar m(f).
\end{gather*}
\end{lemma}

\begin{proof}
Fix $p \in [1,\infty)$ and $t> 2t_0$. It is enough to show both implications for non-negative functions.

\medskip
\enquote{$\Leftarrow$}
Let $0 \leq g \in L^p(M,\mu)$ and $f_n := g \wedge n \in L^{\infty}(M,\mu)$, $n \in \N$. We have  $f_n \uparrow g$ as $n \to \infty$, $\frac{\bar m(U_tf_n)}{\bar m(U_t\I_M)} = \bar m(f_n)$, $n \in \N$, and the implication follows with the monotone convergence theorem.

\medskip
\enquote{$\Rightarrow$}
Let $0 \leq f \in L^{\infty}(M,\mu)$ and set $g_n := \I_{K_n} f  \in L^p(M,\mu)$, $n \in \N$, where $(K_n)$ is an increasing sequence of compact sets such that $\bigcup_{n\in\nat} K_n = M$, see the discussion of topological preliminaries at the beginning of Section~\ref{sec:cpt}. We have $g_n \uparrow f$ as $n \to \infty$, $\frac{\bar m(U_tg_n)}{\bar m(U_t\I_M)} = \bar m(g_n)$, $n \in \N$, and we use monotone convergence once again.
\end{proof}

If \eqref{A3} holds for any $t_0>0$, then it also makes sense to speak about quasi-stationarity on $L^p$, $p<\infty$, for all $t>0$. In particular, the implications in Lemma~\ref{lem:qsm_Lp} hold true for every $t>0$. 

We come back to the measures $m, m^*$ defined in \eqref{eq:quasi_stat_meas}. Here, the condition \eqref{eq:quasi-stat-def} reads
\begin{gather}\label{eq:quasi-stat}
    m(U_tf) = e^{-\lambda_0 t} m(f)
    \quad\text{and}\quad
    m^*(U^*_tf) = e^{-\lambda_0 t} m^*(f), \quad t >0, \;f \in L^{\infty}(M,\mu).
\end{gather}
Since we assume that $\phi_0, \psi_0 \in L^{\infty}(M,\mu)$, \eqref{eq:quasi-stat} remains valid for $f\in L^p(M,\mu)$, $p \in [1,\infty]$. This can be easily seen  without Lemma~\ref{lem:qsm_Lp}. Indeed, it follows from the H\"older inequality and the fact that $\phi_0, \psi_0 \in L^1(M,\mu) \cap L^{\infty}(M,\mu) \subset L^q(M,\mu)$ for any $q \in [1,\infty]$.

Our arguments in Sections~\ref{sec:content} and~\ref{sec:pGSD} are based on the following observation, which is a direct consequence of the estimate in Lemma~\ref{lem:unif_conv}.
\begin{lemma}\label{lem:unif_conv_impr}
    Let \eqref{A0}--\eqref{A3} hold for some $t_0>0$ and denote by $C,\gamma$ the constants appearing in Lemma~\ref{lem:unif_conv}. For every $0 \leq s, r < t$ with $t-s-r > 2 t_0$
    \begin{align} \label{eq:unif_conv_impr}
        \left|e^{\lambda_0 t} u_t(x,y) -  (1/\Lambda) \phi_0(x) \psi_0(y) \right| \leq C \kappa(t,s,r,x,y), \quad \mu\text{-a.e.}\ x,y \in M,
    \end{align}
    holds true with the function
    \begin{gather*}
      \kappa(t,s,r,x,y)
        =
        \begin{cases}
            e^{-\gamma t} &\text{if\ \ }s = r = 0;
            \\[\medskipamount]
        	e^{-\gamma (t-s) } e^{\lambda_0 s} U_{s}\I_M(x)  &\text{if\ \ } s>0, r=0;
            \\[\medskipamount]
        	e^{-\gamma (t-r) } e^{\lambda_0 r} U^*_{r}\I_M(y)  &\text{if\ \ } s=0, r>0;
            \\[\medskipamount]
        	e^{-\gamma (t-s-r) } e^{\lambda_0 s} U_{s}\I_M(x) e^{\lambda_0 r} \, U^*_{r}\I_M(y)  &\text{if\ \ } s>0, r>0.
        \end{cases}
    \end{gather*}
\end{lemma}
\begin{proof}
If $s=r=0$, this is just the estimate from Lemma~\ref{lem:unif_conv}. We will only work out the case $s,r>0$. The remaining two cases follow with very similar arguments. Using the eigenequations
\begin{gather*}
    \phi_0(x) = e^{\lambda_0 s} \int_M u_s(x,z) \phi_0(z)\,\mu(dz)
    \quad\text{and}\quad
    \psi_0(y) = e^{\lambda_0 r} \int_M u_r(w,y) \psi_0(w) \,\mu(dw),
\end{gather*}
and the Chapman-Kolmogorov identities, we see that
\begin{align*}
    &e^{\lambda_0 t} u_t(x,y) - (1/\Lambda) \phi_0(x) \psi_0(y) \\
    &= e^{\lambda_0 s} \int_M u_s(x,z) \left[e^{\lambda_0 (t-s)} u_{t-s}(z,y) - (1/\Lambda) \phi_0(z) \psi_0(y) \right]\mu(dz) \\
    &= e^{\lambda_0 s} \int_M u_s(x,z) \left[e^{\lambda_0 r}\int_M \left[e^{\lambda_0 (t-s-r)} u_{t-s-r}(z,w) - (1/\Lambda) \phi_0(z) \psi_0(w) \right]  u_r(w,y) \,\mu(dw)  \right]\mu(dz).
\end{align*}
By Lemma~\ref{lem:unif_conv} we have
\begin{gather*}
	\left|e^{\lambda_0 (t-s-r)} u_{t-s-r}(z,w) - (1/\Lambda) \phi_0(z) \psi_0(w)\right| \leq C e^{-\gamma (t-s-r)},
\end{gather*}
which gives
\begin{align*}
    &\left|e^{\lambda_0 t} u_t(x,y) - (1/\Lambda) \phi_0(x) \psi_0(y)\right| \\
    &\qquad\leq C e^{-\gamma (t-s-r)} \left(e^{\lambda_0 s} \int_M u_s(x,z)  \,\mu(dz) \right)
	\left(e^{\lambda_0 r} \int_M u_r(w,y) \,\mu(dw)\right).
\end{align*}
This is the claimed bound.
\end{proof}

Our next theorem is the first main result of this section. It relates the exponential quasi-ergodicity of the semigroups $\left\{U_t: t \geq 0 \right\}$ and $\left\{U^*_t: t \geq 0 \right\}$ with the finiteness of their heat content. It also shows that the corresponding quasi-stationary probability distributions are unique. Recall that
\begin{gather*}
    Z(t) := \int U_{t}\I_M(x) \,\mu(dx)
    = \int U^*_{t}\I_M(y) \,\mu(dy)
    = \iint u_{t}(x,y)\,\mu(dx)\,\mu(dy), \quad t>0,
\end{gather*}
is the \emph{heat content} of $\left\{U_t: t \geq 0 \right\}$ or $\left\{U_t^*: t \geq 0 \right\}$. In general, $Z(t) \in (0,\infty]$. For each $t>0$ the condition $Z(t) < \infty$ is equivalent to the boundedness of the operators $U_t, U^*_{t}: L^{\infty}(M,\mu) \to L^1(M,\mu)$.

\bigskip
\begin{center}
\textbf{Unless otherwise stated, we use the value \boldmath $t_0 >0$ \unboldmath from Lemma~\ref{lem:unif_conv_impr} in all further statements as well as for the assumptions \eqref{A1}, \eqref{A3} and \eqref{A4}, and \boldmath $\cut$ \unboldmath denotes the exceptional set defined in \eqref{eq:cut}.}
\end{center}

\begin{theorem}\label{th:hcf_and_eqe}
    Let $M$ be a locally compact Polish space and assume that \eqref{A0}--\eqref{A3} hold. 
	If there exists some $t_1 \geq t_0$ such that the heat content is finite, i.e.\
\begin{gather*}
    Z(t_1) < \infty,
\end{gather*}
then we have the following:
\begin{enumerate}
\item\label{th:hcf_and_eqe-1}
    $\phi_0, \psi_0 \in L^1(M,\mu)$ and both measures $m$ and $m^*$ in \eqref{eq:quasi_stat_meas} are well defined.
\item\label{th:hcf_and_eqe-2}
    For every $p \in [1,\infty]$ there exist $C>0$ such that for every finite measure $\sigma$ on $M$ with $\supp \sigma \cap (M \setminus \cut) \neq \emptyset$ we have for all $t>6t_1$ and $t_1 \leq s \leq t/2$
    \begin{gather*}
        \sup_{\substack{f \in  L^p(M,\mu)\\ \left\|f\right\|_p \leq 1}} \left|\frac{\sigma(U_t f)}{\sigma(U_t \I_M)} - m(f)\right|
        \leq C  \, e^{-\gamma (t-s)} \, \frac{e^{\lambda_0 s}\sigma(U_{s} \I_M)}{\sigma(\phi_0)},
    \intertext{and}
        \sup_{\substack{f \in  L^p(M,\mu)\\ \left\|f\right\|_p \leq 1}} \left|\frac{\sigma(U^*_t f)}{\sigma(U^*_t \I_M)} - m^*(f)\right|
        \leq C  \, e^{-\gamma (t-s)} \, \frac{e^{\lambda_0 s}\sigma(U^*_{s} \I_M)}{\sigma(\psi_0)}.
\end{gather*}
\item\label{th:hcf_and_eqe-3}
    The measures $m$ and $m^*$ are the unique quasi-stationary probability measures 
	of $\left\{U_t: t \geq 0 \right\}$ and $\left\{U^*_t: t \geq 0 \right\}$, respectively, such that $m(M\setminus\cut)\cdot m^*(M\setminus\cut)>0$.
\end{enumerate}
\end{theorem}

\begin{proof}
Recall that $U_t\I_M, U_t^*\I_M\in L^1(M,\mu)$, $t \geq t_1$. Thus, \ref{th:hcf_and_eqe-1} follows from the inequalities  $e^{-\lambda_0 t} \phi_0(x) \leq \left\|\phi_0\right\|_{\infty} U_t\I_M(x)$ and $e^{-\lambda_0 t} \psi_0(x) \leq \left\|\psi_0\right\|_{\infty} U^*_t\I_M(x)$, $x \in M$, $t>0$.

\medskip
We prove~\ref{th:hcf_and_eqe-2} only for the semigroup $\left\{U_t: t \geq 0 \right\}$. The proof for $\left\{U^*_t: t \geq 0 \right\}$ is literally the same. Let $p \in [1,\infty]$ and let $\sigma$ be a finite measure on $M$ such that $\supp \sigma \cap (M \setminus \cut) \neq \emptyset$. Note that $0 < \sigma(\phi_0) \leq e^{\lambda_0 t} \left\|\phi_0\right\|_{\infty} \sigma(U_t \I_M)$ and $0 < \sigma(\psi_0) \leq e^{\lambda_0 t} \left\|\psi_0\right\|_{\infty} \sigma(U^*_t \I_M)$, $t >0$. For every $f \in L^p(M,\mu)$ and $t>6t_1$ (in particular, $t-t_1-s>2t_0$) we may write
\begin{align*}
    \frac{\sigma(U_t f)}{\sigma(U_t \I_M)} - m(f)
    &= \frac{ e^{\lambda_0 t}\sigma(U_t f) - m(f)e^{\lambda_0 t}\sigma(U_t \I_M)}{e^{\lambda_0 t}\sigma(U_t \I_M)}\\
    &= \frac{ e^{\lambda_0 t}\sigma(U_t f) - (1/\Lambda) \, \sigma(\phi_0) \int_M f(y) \psi_0(y) \,\mu(dy)}{e^{\lambda_0 t}\sigma(U_t \I_M)} \\
    &\qquad\mbox{} + m(f) \, \frac{(1/\Lambda) \, \sigma(\phi_0) \int_M\psi_0(y) \,\mu(dy) - e^{\lambda_0 t}\sigma(U_t \I_M)}{e^{\lambda_0 t}\sigma(U_t \I_M)} .
\end{align*}
Since for all $x \in M$,
\begin{align*}
    &e^{\lambda_0 t}U_t f(x) - (1/\Lambda) \, \phi_0(x) \int_M f(y) \psi_0(y) \,\mu(dy)\\
    &\qquad = \int_M f(y) \left(e^{\lambda_0 t}u_t(x,y) - (1/\Lambda) \, \phi_0(x) \psi_0(y)\right) \mu(dy),
\end{align*}
and
\begin{align*}
    &e^{\lambda_0 t}U_t \I_M(x) - (1/\Lambda) \, \phi_0(x) \int_M \psi_0(y) \,\mu(dy)\\
    &\qquad = \int_M \left(e^{\lambda_0 t}u_t(x,y) - (1/\Lambda) \, \phi_0(x) \psi_0(y)\right) \mu(dy),
\end{align*}
we can apply Lemma~\ref{lem:unif_conv_impr} (with $t_1 \leq s \leq t/2$ and $r=t_1$) and integrate with respect to $\sigma$ to get
\begin{align*}
    &\left|e^{\lambda_0 t}\sigma(U_t f) - (1/\Lambda) \, \sigma(\phi_0) \int_M f(y) \psi_0(y) \,\mu(dy)\right|\\
    &\qquad \leq C e^{-\gamma(t-t_1-s)} e^{\lambda_0 (t_1+s)} \sigma(U_{s}\I_M) \int_M |f(y)| U^*_{t_1}\I_M(y) \,\mu(dy),
\end{align*}
and
\begin{align*}
    &\left|e^{\lambda_0 t}\sigma(U_t \I_M) - (1/\Lambda) \, \sigma(\phi_0) \int_M\psi_0(y) \,\mu(dy)\right| \\
    &\qquad\leq C e^{-\gamma(t-t_1-s)} e^{\lambda_0 (t_1+s)} \sigma(U_{s}\I_M) Z(t_1).
\end{align*}
Note that
\begin{align} \label{eq:by_phi}
    e^{-\lambda_0t}\sigma(\phi_0) = \sigma(U_t\phi_0) \leq \|\phi_0\|_\infty\sigma(U_t\I_M)
    \implies
    e^{\lambda_0 t}\sigma(U_t \I_M) \geq \frac{\sigma(\phi_0)}{\left\|\phi_0\right\|_{\infty}}.
\end{align}
This gives
\begin{align*}
\left|\frac{\sigma(U_t f)}{\sigma(U_t \I_M)} - m(f)\right| \leq
           \widetilde C  e^{-\gamma (t-s)} \left(\int_M |f(y)| U^*_{t_1}\I_M(y) \,\mu(dy) + m(|f|) Z(t_1)\right)\frac{e^{\lambda_0 s}\sigma(U_{s}\I_M)}{\sigma(\phi_0)},
\end{align*}
with $\widetilde C := C  e^{(\lambda_0+\eta) t_1} \left\|\phi_0\right\|_{\infty}$, where $C,\gamma$ are the constants from Lemma~\ref{lem:unif_conv_impr}. Finally, by the H\"ol\-der inequality,
\begin{align*}
    \left|\frac{\sigma(U_t f)}{\sigma(U_t \I_M)} - m(f)\right|
    \leq \widetilde C \left(\left\|U^*_{t_1}\I_M\right\|_q + Z(t_1)\frac{\left\|\psi_0\right\|_q}{\left\|\psi_0\right\|_1}\right)
    e^{-\gamma (t-s)}\frac{\sigma(U_{s}\I_M)}{\sigma(\phi_0)} \left\|f\right\|_p,
\end{align*}
where $q$ and $p$ are conjugate exponents. Notice that $\left\|U^*_{t_1}\I_M\right\|_q < \infty$ and $\left\|\psi_0\right\|_q < \infty$. Indeed, if $p=1$, then $q=\infty$, and the claim follows from \eqref{A2} and \eqref{A3} combined with \eqref{eq:eeqs}. If $p=\infty$ and $q=1$, then the claim follows from the fact that $\left\|U^*_{t_1}\I_M\right\|_1 = Z(t_1)< \infty$ and from~\ref{th:hcf_and_eqe-1}. If $1 < p < \infty$, then also $1< q < \infty$, and we have
\begin{gather*}
    \left\|U^*_{t_1}\I_M\right\|_q^q
    = \int_M U^*_{t_1}\I_M(x)(U^*_{t_1}\I_M(x))^{q-1} \,\mu(dx)
    \leq \left\|U^*_{t_1}\I_M\right\|_{\infty}^{q-1} Z(t_1)
    < \infty
\end{gather*}
and
\begin{gather*}
    \left\|\psi_0\right\|_q^q
    = \int_M \psi_0(x)\psi_0^{q-1}(x) \,\mu(dx)
    \leq \left\|\psi_0\right\|_{\infty}^{q-1} \left\|\psi_0\right\|_1
    < \infty.
\end{gather*}
This completes the proof of~\ref{th:hcf_and_eqe-2}.

In order to show~\ref{th:hcf_and_eqe-3}, assume that both $m$ and $\bar m$ are quasi-stationary probability measures and $\supp \bar m \cap (M \setminus \cut) \neq \emptyset$. Since $\bar m$ is finite, we can use $\bar m = \sigma$ in Part~\ref{th:hcf_and_eqe-2} and we get, because of the quasi-stationary property \eqref{eq:quasi-stat-def},
\begin{gather*}
        0 = \lim_{t\to\infty}\left|\frac{\bar m(U_t f)}{\bar m(U_t \I_M)} - m(f)\right|
        = \left|\bar m(f) - m(f)\right|
\end{gather*}
(uniformly) for all $f\in L^\infty(M,\mu)$ with $\|f\|_{\infty}\leq 1$. By our assumptions, $\mu$ has full topological support, it is locally finite -- hence finite on all compact sets -- and the compact sets generate the Borel $\sigma$-algebra on $M$. Moreover, $M$ is $\sigma$-compact, i.e.\ there is an increasing sequence of compact sets $K_n$ such that $M = \bigcup_{n\geq 1}K_n$. Taking $f=\I_K$ for any compact set $K$, we conclude that $m = \bar m$ first on the compact sets and, using the standard uniqueness theorem for measures, on all Borel sets. The proof of the uniqueness of $m^*$ is similar.
\end{proof}

\begin{corollary}\label{cor:hcf_uqe}
   Let $M$ be a locally compact Polish space and assume that \eqref{A0}--\eqref{A3} hold, 
	and assume that there exists some $t_1 \geq t_0$ such that $Z(t_1) < \infty$.
\begin{enumerate}
\item\label{cor:hcf_uqe-a}
    Both semigroups $\left\{U_t: t \geq 0 \right\}$ and $\left\{U^*_t: t \geq 0 \right\}$ are quasi-ergodic with exponential time rate $e^{-\gamma t}$ and  space rates $U_{t_1}\I_M(x)/\phi_0(x)$ and $U^*_{t_1}\I_M(x)/\psi_0(x)$, respectively. More precisely, for every $p \in [1,\infty]$ there is a constant $C>0$ such that for every $f \in L^p(M,\mu)$ with $\left\|f\right\|_p \leq 1$
    \begin{gather*}
        \left|\frac{U_t f(x)}{U_t \I_M(x)} - m(f)\right|
        \leq C  \, e^{-\gamma t} \, \frac{U_{t_1} \I_M(x)}{\phi_0(x)},
        \quad t >6t_1, \; x \in M \setminus \cut,
    \end{gather*}
    and
    \begin{gather*}
        \left|\frac{U^*_t f(x)}{U^*_t \I_M(x)} - m^*(f)\right|
        \leq C  \, e^{-\gamma t} \, \frac{U^*_{t_1} \I_M(x)}{\psi_0(x)},
        \quad t > 6t_1, \; x \in M \setminus \cut.
    \end{gather*}
    Moreover,
    the measures $m$, and $m^*$, are the only quasi-stationary measures of $\left\{U_t: t \geq 0 \right\}$, and $\left\{U^*_t: t \geq 0 \right\}$, respectively, such that $m(M\setminus\cut)\cdot m^*(M\setminus\cut)>0$.

\item\label{cor:hcf_uqe-b}
    Since $U_t \I_M, U^*_t \I_M \in C_{b}(M)$ and $\phi_0$, $\psi_0$ are continuous and strictly positive on $M$, the exponential quasi-ergodicity  in~\ref{cor:hcf_uqe-a} is uniform for all $x$ in compact subsets of $M \setminus \cut$. 
\end{enumerate}
\end{corollary}
If $\mu(M) < \infty$, then the heat content is automatically finite for large times, and we recover some of the results from \cite{bib:ZSS}.

\section{Progressive ground state domination and heat content} \label{sec:pGSD_and_hc}

In this section we will need the notions of \emph{hemi-compact space} and \emph{exhausting} family of sets.

\begin{definition}\label{def:ex_family}
The space $M$ is said to be \emph{hemi-compact} if there exists an \emph{exhausting family of sets} in $M$, i.e.\ a family
$\left\{K_t : t \geq t_1\right\}$ (defined for some $t_1 \geq 0$) of compact subsets of $M$ such that:
\begin{enumerate}
\item
$K_t \subset K_s$, for every $t_1 \leq t < s$;
\item
$\bigcup_{t \geq t_1} K_t = M$;
\item
for every compact set $K \subset M$ there exists some $t \geq t_1$ such that $K \subset K_t$.
\end{enumerate}
\end{definition}
Since $M$ is a locally compact Polish space, $M$ is hemi-compact; moreover the concepts of $\sigma$-compactness and hemi-compactness coincide in such spaces, see e.g.\ \cite[p.~195]{bib:E}. If $M$ is compact, it is also hemi-compact with a trivial exhausting family $\left\{K_t : t \geq t_1\right\}$ such that $K_t=M$ for all $t \geq t_1$. Assume now that \eqref{A0}--\eqref{A3} hold with some $t_0>0$. The following properties will be central to our investigations in this section.

\begin{definition}[aGSD and pGSD]\label{def:pGSD}
    Let $\left\{U_t: t \geq 0 \right\}$ be a semigroup of operators on $L^2(M,\mu)$ satisfying the conditions \eqref{A0}--\eqref{A3} and let $\lambda_0$ and $\phi_0$ be the corresponding ground state eigenvalue and normalized eigenfunction.
    \begin{enumerate}
    \item\label{def:pGSD-a}
    The semigroup $\left\{U_t: t \geq 0 \right\}$ is said to be \emph{asymptotically ground state dominated} (aGSD, for short), if
    there exist constants $C, t_1>0$ such that
    \begin{align}\label{eq:aGSD}
        U_t \I_M(x) \leq C e^{-\lambda_0 t} \phi_0(x), \quad t \geq t_1, \; x \in M.
    \end{align}
    \item\label{def:pGSD-b}
    The semigroup $\left\{U_t: t \geq 0 \right\}$ is said to be \emph{progressively ground state dominated} (pGSD, for short), if
    there exist constants $C, t_1>0$ and an exhausting family of compact sets $\left\{K_t : t \geq t_1\right\}$ such that
    \begin{align}\label{eq:pGSD}
        \I_{K_t}(x) \cdot U_t \I_M(x)  \leq C e^{-\lambda_0 t} \phi_0(x), \quad t \geq t_1, \; x \in M.
    \end{align}
    \end{enumerate}
\end{definition}

\begin{remark} \label{rem:opposite}
\begin{enumerate}
\item\label{rem:opposite-a}
    Because of the eigenequation $e^{-\lambda_0 t} \phi_0 = U_t \phi_0$ there is some kind of reverse inequality to \eqref{eq:aGSD}: $e^{-\lambda_0 t} \phi_0(x) \leq  \left\|\phi_0\right\|_{\infty} U_t\I_M(x)$, $x \in M$, $t>0$.

\item\label{rem:opposite-x}
    Notice that both $U_t\I_M(x)$ and $\phi_0(x)$ are continuous functions. Since $\mu(\cut)=0$ and since $\mu$ charges every open set, the interior of $\cut$ is empty and $\overline{M\setminus\cut}=M$. This means that we can replace in \eqref{eq:aGSD} and \eqref{eq:pGSD} the condition \enquote{$x\in M$} by \enquote{$x\in M\setminus\cut$}.

\item\label{rem:opposite-b}
    aGSD was studied in the papers \cite{bib:KL15a,bib:KKL2018} in relation to intrinsic ultracontractivity and hypercontractivity. It is shown in \cite{bib:KSchi20}
    that aGSD is a much more restrictive condition than the finiteness of heat content for large times. It follows directly from Lemma~\ref{lem:unif_conv_impr} that aGSD of $\left\{U_t: t \geq 0 \right\}$ implies $\phi_0 \in L^1(M,\mu)$ as well as the finiteness of the heat content for large times; this is also true for $\left\{U^*_t: t \geq 0 \right\}$ and $\psi_0$.

\item\label{rem:opposite-c}
    pGSD was recently introduced in \cite{bib:KSchi20} as a direct consequence of the concept of progressive intrinsic ultracontractivity (pIUC) for Schr\"odinger semigroups. In this and the next section we will explain the impact of pGSD on the quasi-ergodic behaviour of compact semigroups and discuss the relation between pGSD and the finiteness of the heat content.
		
\item\label{rem:opposite-d}
    From the definition we see that aGSD always implies pGSD. Recall that $U_t\I_M$, $U^*_t\I_M$, $\phi_0$, $\psi_0  \in C_b(M)$. If the space $M$ is compact and $\phi_0$, $\psi_0$ are strictly positive everywhere (i.e.\ $\cut = \emptyset$), then $\inf_M \phi_0 > 0$ and $\inf_M \psi_0 > 0$, so aGSD always holds. In particular, pGSD holds with an exhausting family $\left\{K_t : t \geq t_1\right\}$ such that $K_t=M$ for all $t \geq t_1$.
\end{enumerate}
\end{remark}

We first show that aGSD of $\left\{U_t: t \geq 0 \right\}$  implies the exponential uniform (for the whole space $M$) quasi-ergodicity of this semigroup on every $L^p(M,\mu)$, $p \geq 1$; as before, this also holds for $\left\{U^*_t: t \geq 0 \right\}$. In particular, this gives a relatively short and elementary proof of the main results of \cite[Theorem 1, Corollary 2]{bib:KP} in our framework. We also have  some converse statement.

\begin{corollary}\label{cor:agsd-euqe}
   Let $M$ be a general locally compact \textup{(}but not necessarily compact\textup{)} Polish space and assume that \eqref{A0}--\eqref{A3} hold.
    Then we have the following assertions.
	\begin{enumerate}
	\item\label{cor:agsd-euqe-a}
    If the semigroup $\left\{U_t: t \geq 0 \right\}$ is aGSD for some $t_1 \geq t_0$, then there is for every $p \in [1,\infty]$ a constant $C>0$, such that for every $f \in L^p(M,\mu)$ with $\left\|f\right\|_p \leq 1$
    \begin{gather}\label{eq:agsd}
       \sup_{x \in M \setminus \cut} \left|\frac{U_t f(x)}{U_t \I_M(x)} - m(f)\right|
        \leq C  \, e^{-\gamma t},
        \quad t >6t_1.
    \end{gather}
	\item\label{cor:agsd-euqe-b}
	If there is a constant $C>0$ such that
    \begin{gather*}
        \sup_{x \in M \setminus \cut} \left|\frac{U_t \phi_0(x)}{U_t \I_M(x)} - C\right| \to 0
            \quad\text{as}\quad t \to \infty,	
    \end{gather*}	
    then the semigroup $\left\{U_t: t \geq 0 \right\}$ is aGSD.
	\end{enumerate}
   Both assertions~\ref{cor:agsd-euqe-a} and~\ref{cor:agsd-euqe-b} are also true for the semigroup $\left\{U^*_t: t \geq 0 \right\}$ and the quasi-stationary measure $m^*$. Under aGSD,
	the measures $m$, and $m^*$, are the only quasi-stationary measures of $\left\{U_t: t \geq 0 \right\}$, and $\left\{U^*_t: t \geq 0 \right\}$, respectively, such that $m(M\setminus\cut)\cdot m^*(M\setminus\cut)>0$.
\end{corollary}

\begin{proof}
Assertion~\ref{cor:agsd-euqe-a} follows directly from a combination of Remark~\ref{rem:opposite}.\ref{rem:opposite-b} and Corollary~\ref{cor:hcf_uqe}.\ref{cor:hcf_uqe-a}. Let us show~\ref{cor:agsd-euqe-b}. Because of the eigenequation $U_t \phi_0 = e^{-\lambda_0 t} \phi_0$, we see that there is some $t_1>0$ such that
\begin{gather*}
    \left|\frac{e^{-\lambda_0 t} \phi_0(x)}{U_t \I_M(x)} - C\right| \leq \frac{C}{2}, \quad t \geq t_1, \ x \in M \setminus \cut,
\end{gather*}
which implies
\begin{gather*}
    \frac{e^{-\lambda_0 t} \phi_0(x)}{U_t \I_M(x)} \geq \frac{C}{2}, \quad t \geq t_1, \ x \in M \setminus \cut.
\end{gather*}
Since $C >0$ and both functions $\phi_0$ and $U_t \I_M$ are bounded and continuous, this completes the proof.
\end{proof}

As explained in Remark~\ref{rem:opposite}.\ref{rem:opposite-d}, if the space $M$ is compact (hence, $\mu(M) <\infty$, due to the local finiteness of $\mu$) and the ground states are everywhere strictly positive, the semigroups $\left\{U_t: t \geq 0 \right\}$, $\left\{U^*_t: t \geq 0 \right\}$ are automatically aGSD, and the heat content is finite for large times. In particular, Corollary~\ref{cor:agsd-euqe} implies that both semigroups are exponentially uniformly (on the whole space $M$) quasi-ergodic.

If the space $M$ is non-compact, then aGSD is a rather restrictive condition. However, as illustrated in \cite{bib:KSchi20}, one should expect that in this case a wide range of non-aGSD compact semigroups are still pGSD. From now on we focus on the quasi-ergodicity of pGSD semigroups in the setting of non-compact spaces.

We first show that the finiteness of the heat content implies pGSD for both $\left\{U_t: t \geq 0 \right\}$ and $\left\{U^*_t: t \geq 0 \right\}$; consequently, we get progressive uniform quasi-ergodicity with an exponential time-rate. To do so, we need a lemma. To keep the proofs simple, we assume in Lemma~\ref{lem:eta_def} and Theorem~\ref{th:pGSD} that $\cut = \emptyset$, i.e.\ that the ground states are everywhere strictly positive.

\begin{lemma} \label{lem:eta_def}
    Let $M$ be a locally compact, but not compact, Polish space. Assume that \eqref{A0}--\eqref{A4} hold, $\cut = \emptyset$, and
    let $\left\{K_t : t \geq t_0\right\}$ be an exhausting family of compact sets in $M$. Set
\begin{gather*}
    h(t) := \min\left\{\inf_{x \in K_t} \phi_0(x),\inf_{x \in K_t} \psi_0(x) \right\}, \quad t \geq t_0,
\end{gather*}
and consider the generalized inverse of $h$,
\begin{gather*}
    h^{-1}(s) := \inf\left\{t \geq t_0: h(t) < s\right\}, \quad s \in (0, h(t_0)].
\end{gather*}
Define the function $\eta: [-\gamma^{-1} \log h(t_0),\infty) \to [t_0,\infty)$ by
\begin{align}\label{eg:eta_def}
\eta(t) := h^{-1}(s)\big{|}_{s=e^{-\gamma t}} = h^{-1}(e^{-\gamma t}).
\end{align}
Then $\eta$ is increasing on $[-\gamma^{-1}\log h(t_0),\infty)$, $\eta(t) \to \infty$ as $t \to \infty$ and
\begin{gather*}
e^{\gamma t} \min\left\{\inf_{x \in K_{\eta(t)}} \phi_0(x),\inf_{x \in K_{\eta(t)}} \psi_0(x) \right\} = e^{\gamma t} h(\eta(t)) = 1,  \qquad t \geq -\gamma^{-1}\log h(t_0).
\end{gather*}
\end{lemma}

\begin{proof}
    Because of the eigenequations, our assumptions \eqref{A1}--\eqref{A3} and $\cut=\emptyset$, the ground state eigenfunctions $\phi_0$ and $\psi_0$ are continuous and strictly positive on $M$. Together with the inequality in Remark~\ref{rem:opposite}.\ref{rem:opposite-a} and the assumption \eqref{A4} this implies that the function $h$ is continuous, strictly positive and decreasing on $[t_0,\infty)$ such that $h(t) \to 0$ as $t \uparrow \infty$. In particular, its generalized inverse function $h^{-1}$ is decreasing on $(0,h(t_0)]$, $h^{-1}(s) \to \infty$ as $s \downarrow 0$, and we have $h(h^{-1}(s)) = s$ for every $s \in (0,h(t_0)]$. Therefore, $\eta$ is increasing on $[-\gamma^{-1}\log h(t_0),\infty)$ as it is a composition of two decreasing functions $h^{-1}$ and $t \mapsto e^{-\gamma t}$, $\eta(t) \to \infty$ as $t \to \infty$, and $e^{\gamma t} h(\eta(t)) = e^{\gamma t} h(h^{-1}(e^{-\gamma t})) = e^{\gamma t} e^{-\gamma t} = 1$. This completes the proof.
\end{proof}

In the following theorem we will use the notation $\cM^1(K)$ to denote all probability measures $\sigma$ on $(M,\cB(M))$ such that $\supp\sigma\subset K$. Recall that we already know from Theorem~\ref{th:hcf_and_eqe}, that the finiteness of the heat content for large times implies that $\phi_0, \psi_0 \in L^1(M,\mu)$ and that the measures $m$, and $m^*$, defined in \eqref{eq:quasi_stat_meas} are the only quasi-stationary measures on each $L^p(M,\mu)$, $p \in [1,\infty]$, of the semigroups $\left\{U_t: t \geq 0 \right\}$, and $\left\{U^*_t: t \geq 0 \right\}$, respectively.

\begin{theorem}\label{th:pGSD}
   Let $M$ be a locally compact, but not compact, Polish space. Assume that \eqref{A0}--\eqref{A4} hold, 
    let $\left\{K_t : t \geq t_0\right\}$ be an exhausting family of compact sets in $M$, and assume that $\cut = \emptyset$. Set $\widetilde t_0 = \max\left\{0, -\gamma^{-1}\log h(t_0)\right\}$ and let $\eta$ be the function defined in \eqref{eg:eta_def}. If there is some $t_1 \geq t_0$ such that the heat content $Z(t_1) < \infty$, then the following assertions hold.
\begin{enumerate}
\item\label{th:pGSD-a}
    Both semigroups $\left\{U_t: t \geq 0 \right\}$ and $\left\{U^*_t: t \geq 0 \right\}$ are pGSD for the exhausting family of sets $\left\{K_{\eta(t)} : t \geq \widetilde t_0\right\}$.
\item\label{th:pGSD-b}
    Both semigroups $\left\{U_t: t \geq 0 \right\}$ and $\left\{U^*_t: t \geq 0 \right\}$ are exponentially progressively uniformly quasi-ergodic, i.e.\ for every $p \in [1,\infty]$, there exists some $C>0$, such that
    \begin{gather*}
        \sup_{\sigma \in \cM^1(K_{\eta(t/2)})} \sup_{\substack{f \in  L^p(M,\mu)\\ \left\|f\right\|_p \leq 1}} \left|\frac{\sigma(U_t f)}{\sigma(U_t \I_M)} - m(f)\right|
        \leq C  \, e^{-(\gamma/2) t}, \quad t >6t_1,
    \end{gather*}
    and
    \begin{gather*}
        \sup_{\sigma \in \cM^1(K_{\eta(t/2)})} \sup_{\substack{f \in  L^p(M,\mu)\\ \left\|f\right\|_p \leq 1}} \left|\frac{\sigma(U^*_t f)}{\sigma(U^*_t \I_M)} - m^*(f)\right|
        \leq C  \, e^{-(\gamma/2) t}, \quad t >6t_1.
    \end{gather*}
\end{enumerate}
\end{theorem}
\begin{proof}
    Recall that $\psi_0(x) \leq \left\|\psi_0\right\|_{\infty} e^{\lambda_0 t} U^*_t\I_M(x)$ and $\phi_0(x) \leq \left\|\phi_0\right\|_{\infty}e^{\lambda_0 t}  U_t\I_M(x)$. Since $Z(t_1) < \infty$, this yields $\phi_0, \psi_0 \in L^1(M,\mu)$. Therefore, Lemma~\ref{lem:unif_conv_impr} (with $s = 0$ and $r=t_1$) and Lemma~\ref{lem:eta_def} give
\begin{align*}
    \left|e^{\lambda_0 t} U_t\I_M(x) - (1/\Lambda) \phi_0(x) \left\|\psi_0\right\|_1 \right|
    &\leq \int_M\left|e^{\lambda_0 t} u_t(x,y) - (1/\Lambda) \phi_0(x) \psi_0(y) \right|\mu(dy) \\
    &\leq C e^{-\gamma (t-t_1)} Z(t_1) \\
	&\leq C e^{\gamma t_1}Z(t_1) \frac{\phi_0(x)}{e^{\gamma t} \inf_{x \in K_{\eta(t)}} \phi_0(x)} \\
    &\leq c_1 \phi_0(x),
\end{align*}
for every $x \in K_{\eta(t)}$ and sufficiently large $t$; since the function $\eta$ is increasing and $\eta(t) \to \infty$ as $t \to \infty$, the family $\left\{K_{\eta(t)} : t \geq \widetilde t_0\right\}$ is indeed exhausting. This gives
\begin{gather*}
    U_t\I_M(x) \leq (c_1+(1/\Lambda) \left\|\psi_0\right\|_1) e^{-\lambda_0 t} \phi_0(x), \quad x \in K_{\eta(t)},
\end{gather*}
for sufficiently large $t$, showing that the semigroup $\left\{U_t: t \geq 0 \right\}$ is pGSD. The pGSD property for the semigroup $\left\{U^*_t: t \geq 0 \right\}$ follows with exactly the same ergument. This completes the proof of~\ref{th:pGSD-a}.

Part~\ref{th:pGSD-b} follows from Theorem~\ref{th:hcf_and_eqe}.\ref{th:hcf_and_eqe-2} with $s = t/2$ and probability measures $\sigma$ supported in $K_{\eta(t/2)}$, by an application of the pGSD property which we have already established in Part~\ref{th:pGSD-a}.
\end{proof}

\section{Progressive ground state domination and quasi-ergodicity} \label{sec:pGSD}

As we have already noticed in the previous section, in a non-compact space $M$ aGSD is stronger than the finiteness of heat content for large times. However, the finiteness of the heat content for large times implies pGSD with a certain exhausting family of sets. Recall that these properties -- aGSD, pGSD, finite heat content -- automatically hold in compact spaces. Consequently, we also have uniform quasi-ergodicity at an exponential time-rate, see Remark~\ref{rem:opposite}.\ref{rem:opposite-d}, Corollary~\ref{cor:agsd-euqe}, and the comments following the corollary.

On the other hand, see \cite{bib:KSchi20}, pGSD is a general regularity property for compact non-aGSD semigroups, which can be studied on its own; in particular, pGSD is still true for a wide range of compact semigroups having infinite heat content for large times. In the present section we follow this path, introducing the concept of progressive uniform quasi-ergodicity and studying its relation to pGSD.

Let $a,b \in (0,1)$ be arbitrary numbers such that $a+2b=1$. We show first that in a noncompact space $M$ and under \eqref{A4}, pGSD for the exhausting family $\left\{K_t : t \geq t_0\right\}$ implies progressive uniform quasi-ergodicity  for the exhausting family $\left\{K_{at} : t \geq t_0/a\right\}$; the time-rate is given by
\begin{gather} \label{eq:kappa-rate}
    \kappa_b(t):= e^{-\gamma b t} + \sup_{x \in K_{bt}^c} U_{t_0}\I_M(x) + \sup_{x \in K_{bt}^c} U^*_{t_0}\I_M(x).
\end{gather}
Note that  $\kappa_b(t) \downarrow 0$ as $t \uparrow \infty$. Quite often, see the examples, $\kappa_b$ decays exponentially at infinity.

We have to make sure that $\phi_0, \psi_0 \in L^1(M,\mu)$, which means that both measures $m$ and $m^*$ in \eqref{eq:quasi_stat_meas} are well-defined. If the space $M$ is compact, this follows from $\mu(M) < \infty$ and $\phi_0, \psi_0 \in L^\infty(M,\mu)$ due to \eqref{A3}. For locally compact (but non-compact) Polish spaces $M$, this is still true under aGSD, see Remark~\ref{rem:opposite}.\ref{rem:opposite-b}. Now we show that $\phi_0, \psi_0 \in L^1(M,\mu)$ holds if both semigroups $\left\{U_t: t \geq 0 \right\}$ and $\left\{U^*_t: t \geq 0 \right\}$ enjoy the pGSD property.

\begin{lemma}\label{lem:pGSD_L1}
    Let $M$ be a locally compact, but not compact, Polish space and assume \eqref{A0}--\eqref{A3}. If the semigroups $\left\{U_t: t \geq 0 \right\}$ and $\left\{U^*_t: t \geq 0 \right\}$ are pGSD, then $\phi_0, \psi_0 \in L^1(M,\mu)$.
\end{lemma}
\begin{proof}
    Suppose that pGSD holds for the exhausting family of compact sets $\left\{K_t : t \geq t_0\right\}$. We only show that $\phi_0 \in L^1(M,\mu)$. The proof for $\psi_0$ is similar. By Lemma~\ref{lem:unif_conv_impr} (with $t_0 < s < t$ such that $t-s > 2t_0$ and $r=0$) and pGSD of $\left\{U_t: t \geq 0 \right\}$, we have for $\mu$-almost all $x \in K_s$ and $y \in K_s \setminus \cut$
\begin{align*}
    e^{\lambda_0 t} u_t(x,y)
        &\geq (1/\Lambda) \phi_0(x) \psi_0(y)-C e^{-\gamma (t-s)} U_s \I_M(x) \\
        &\geq \left((1/\Lambda) \psi_0(y)-c_1 e^{-\gamma (t-s) - \lambda_0 s} \right) \phi_0(x).
\end{align*}
We integrate on both sides of the inequality in $x \in K_s$ and use pGSD of $\left\{U^*_t: t \geq 0 \right\}$ to get
\begin{gather*}
    \int_{K_s} \phi_0(x)\,\mu(dx)
    \leq \frac{e^{\lambda_0 t} U^*_t \I_M(y)}{(1/\Lambda) \psi_0(y)-c_1 e^{-\gamma (t-s) - \lambda_0 s}}
    \leq \frac{c_2 \psi_0(y)}{(1/\Lambda) \psi_0(y)-c_1 e^{-\gamma (t-s) - \lambda_0 s}},
\end{gather*}
for sufficiently large $t$. Letting $t \to \infty$ we obtain
\begin{gather*}
    \int_{K_s} \phi_0(x) \,\mu(dx)
    \leq c_2 \Lambda\quad\text{for all\ \ } s > t_0,
\end{gather*}
and monotone convergence finally proves $\phi_0 \in L^1(M,\mu)$.
\end{proof}

\begin{theorem} \label{th:pGSD-to-puqe}
    Let $M$ be a locally compact, but not compact, Polish space. Assume that \eqref{A0}--\eqref{A4} hold 
    and that both semigroups $\left\{U_t: t \geq 0 \right\}$ and $\left\{U^*_t: t \geq 0 \right\}$ are pGSD for the exhausting family of compact sets $\left\{K_t : t \geq t_0\right\}$. Then we have $\phi_0, \psi_0 \in L^1(M,\mu)$ and, for every $a, b \in (0,1)$ such that $a+2b=1$, both semigroups $\left\{U_t: t \geq 0 \right\}$ and $\left\{U^*_t: t \geq 0 \right\}$ are progressively uniformly quasi-ergodic for the exhausting family $\left\{K_{at} : t \geq t_0/a\right\}$ with time-rate $\kappa_b(t)$ given by \eqref{eq:kappa-rate}, i.e.\ for every $p \in [1,\infty]$ there exists some $C>0$ such that for all $t>(4/(a \wedge b)) t_0$
    \begin{gather*}
            \sup_{\sigma \in \cM^1(K_{a t}) \atop \supp \sigma \cap (M \setminus \cut) \neq \emptyset} \sup_{f \in  L^p(M,\mu) \atop \left\|f\right\|_p \leq 1} \left|\frac{\sigma(U_t f)}{\sigma(U_t \I_M)} - m(f)\right|
        \leq C  \kappa_b(t),
    \intertext{and}
        \sup_{\sigma \in \cM^1(K_{at}) \atop \supp \sigma \cap (M \setminus \cut) \neq \emptyset} \sup_{f \in  L^p(M,\mu) \atop \left\|f\right\|_p \leq 1} \left|\frac{\sigma(U^*_t f)}{\sigma(U^*_t \I_M)} - m^{*}(f)\right|
        \leq C  \kappa_b(t).
    \end{gather*}
\end{theorem}
\begin{proof}
    In view of Lemma~\ref{lem:pGSD_L1} we have $\phi_0, \psi_0 \in L^1(M,\mu)$. We proceed as in the proof of Theorem~\ref{th:hcf_and_eqe}. Since the proof for the dual semigroup is similar, we restrict our attention to $\left\{U_t: t \geq 0 \right\}$. Fix $p \in [1,\infty]$. For every $f \in L^p(M,\mu)$, all $t>(4/(a \wedge b)) t_0$ and $\sigma \in \cM^1(K_{at})$ such that $\supp \sigma \cap (M \setminus \cut) \neq \emptyset$ we have
    \begin{gather}
    \label{eq:starting_est}
    \begin{aligned}
    \frac{\sigma(U_t f)}{\sigma(U_t \I_M)} - m(f)
    &= \frac{ e^{\lambda_0 t}\sigma(U_t f) - m(f)e^{\lambda_0 t}\sigma(U_t \I_M)}{e^{\lambda_0 t}\sigma(U_t \I_M)}\\
    &= \frac{ e^{\lambda_0 t}\sigma(U_t f) - (1/\Lambda) \, \sigma(\phi_0) \int_M f(y) \psi_0(y) \,\mu(dy)}{e^{\lambda_0 t}\sigma(U_t \I_M)} \\
    &\qquad\mbox{} + m(f) \, \frac{(1/\Lambda) \, \sigma(\phi_0) \int_M\psi_0(y) \,\mu(dy) - e^{\lambda_0 t}\sigma(U_t \I_M)}{e^{\lambda_0 t}\sigma(U_t \I_M)} .
    \end{aligned}
    \end{gather}
as well as
    \begin{gather}\label{eq:pgsd_1}
    \begin{aligned}
    &e^{\lambda_0 t}U_t f(x) - (1/\Lambda) \, \phi_0(x) \int_M f(y) \psi_0(y) \,\mu(dy) \\
    &\quad= \int_M \left(e^{\lambda_0 t}u_t(x,y) - (1/\Lambda) \, \phi_0(x) \psi_0(y)\right) f(y) \,\mu(dy),
    \end{aligned}
    \intertext{and}\label{eq:pgsd_2}
    \begin{aligned}
    &e^{\lambda_0 t}U_t \I_M(x) - (1/\Lambda) \, \phi_0(x) \int_M \psi_0(y) \,\mu(dy) \\
    &\quad= \int_M \left(e^{\lambda_0 t}u_t(x,y) - (1/\Lambda) \, \phi_0(x) \psi_0(y)\right) \mu(dy),
    \end{aligned}
    \end{gather}
We first estimate \eqref{eq:pgsd_1}. With the Chapman--Kolmogorov equation and the eigenequation, we get
\begin{align*}
    &e^{\lambda_0 t}U_t f(x) - (1/\Lambda) \, \phi_0(x) \int_M f(y) \psi_0(y) \,\mu(dy) \\
    &\quad = \int_M \left[\left(e^{\lambda_0 (t-2t_0)}u_{t-2t_0}(x,w) - (1/\Lambda) \, \phi_0(x) \psi_0(w)\right) e^{2\lambda_0 t_0} \int_M u_{2t_0}(w,y)f(y) \,\mu(dy)\right]\mu(dw) \\
	&\quad= \int_{K_{bt}} \Big( \dots \Big) e^{2\lambda_0 t_0} U_{2t_0} f(w) \,\mu(dw)
		                + \int_{K_{bt}^c} \Big( \dots \Big) e^{2\lambda_0 t_0} U_{2t_0} f(w)\,\mu(dw) .
\end{align*}
Since
\begin{gather*}
    \left|U_{2t_0} f(x) \right| \leq \left\|U_{t_0} \right\|_{p,\infty} \left\|f\right\|_p U_{t_0} \I_M(x), \quad x \in M,
\end{gather*}
we obtain
\begin{align*} 
    &\left|e^{\lambda_0 t}U_t f(x) -  (1/\Lambda) \, \phi_0(x) \int_M f(y) \psi_0(y) \,\mu(dy) \right|  \\
    &\quad\leq  e^{2\lambda_0 t_0} \left\|U_{t_0} \right\|_{p,\infty} \left\|f\right\|_p \left(\left\|U_{t_0} \I_M\right\|_{\infty} I_1(t,x) + \sup_{w \in K_{bt}^c} U_{t_0} \I_M(w) I_2(t,x) \right),
\end{align*}
where
\begin{gather*}
    I_1(t,x):= \int_{K_{bt}} \left|e^{\lambda_0 (t-2t_0)}u_{t-2t_0}(x,w) - (1/\Lambda) \, \phi_0(x) \psi_0(w)\right| \mu(dw),
\intertext{and}
    I_2(t,x):= \int_{K_{bt}^c} \left|e^{\lambda_0 (t-2t_0)}u_{t-2t_0}(x,w) - (1/\Lambda) \, \phi_0(x) \psi_0(w)\right| \mu(dw).
\end{gather*}
By Lemma~\ref{lem:unif_conv_impr} (with $s = at$ and $r=bt$; in particular $t-2t_0-s-r= bt - 2t_0 > 2t_0$) and pGSD, we see that for all $x \in K_{at}$
\begin{gather*}
    I_1(t,x)
    \leq e^{-\gamma (bt-2t_0)} e^{\lambda_0 a t} U_{at}\I_M(x) e^{\lambda_0 b t} \int_{K_{bt}} U^*_{bt}\I_M(w) \,\mu(dw)
    \leq 
    c_1(t_0) e^{-\gamma bt} \phi_0(x) \left\|\psi_0\right\|_1.
\end{gather*}
Again by pGSD, we get for $x \in K_{at} \subset K_{t-2t_0}$
\begin{gather*}
    I_2(t,x)
    \leq e^{\lambda_0 (t-2t_0)} U_{t-2t_0}\I_M(x) + (1/\Lambda) \, \phi_0(x) \left\|\psi_0\right\|_1
    \leq (c_2 + (1/\Lambda) \left\|\psi_0\right\|_1) \phi_0(x).
\end{gather*}
Together this gives, for any measure $\sigma \in  \cM^1(K_{at})$ with $\supp \sigma \cap (M \setminus \cut) \neq \emptyset$, the estimate
\begin{align*}
    \left|e^{\lambda_0 t}\sigma(U_t f) - (1/\Lambda) \, \sigma(\phi_0) \int_M f(y) \psi_0(y) \,\mu(dy) \right|
    \leq c_3 \left(e^{-\gamma b t} + \sup_{w \in K_{bt}^c} U_{t_0} \I_M(w)\right) \left\|f\right\|_p \sigma(\phi_0).
\end{align*}
The estimate for \eqref{eq:pgsd_2} follows if we apply the above estimate for $p=\infty$ to $f = \I_M \in L^{\infty}(M,\mu)$, i.e.\ we get for all $\sigma \in  \cM^1(K_{at})$, $\supp \sigma \cap (M \setminus \cut) \neq \emptyset$,
\begin{gather*}
    \left|e^{\lambda_0 t}\sigma(U_t \I_M) - (1/\Lambda) \, \sigma(\phi_0) \left\|\psi_0\right\|_1 \right|
    \leq c_4 \left(e^{-\gamma b t} + \sup_{w \in K_{bt}^c} U_{t_0} \I_M(w)\right) \sigma(\phi_0).
\end{gather*}
It is important that the constants $c_3$ and $c_4$ do not depend on $\sigma$.
Finally, using \eqref{eq:starting_est}, \eqref{eq:by_phi}, and the above estimates, we arrive at
\begin{gather*}
    \left|\frac{\sigma(U_t f)}{\sigma(U_t \I_M)} - m(f)\right|
    \leq c_4 \left(e^{-\gamma b t} + \sup_{w \in K_{bt}^c} U_{t_0} \I_M(w)\right) \left\|f\right\|_p,
\end{gather*}
which is valid for all measures $\sigma \in \cM^1(K_{at})$ such that $\supp \sigma \cap (M \setminus \cut) \neq \emptyset$. This completes the proof.
\end{proof}

\begin{remark} \label{rem:optimal_theta}
    In Theorem~\ref{th:pGSD-to-puqe} we prove that pGSD with the exhausting family $\left\{K_t : t \geq t_0\right\}$ implies progressive quasi-ergodicity with $\left\{K_{at} : t \geq t_0/a\right\}$, for an arbitrary $a \in (0,1)$. We will see in Section~\ref{subsec:HO} that this result is sharp in the sense that one cannot expect progressive quasi-ergodicity which is uniform for the original exhausting family $\{K_t : t \geq t_0\}$.
\end{remark}

Our next theorem shows that pGSD is, in fact, equivalent to a version of progressive uniform quasi-ergodicity, cf.\ Corollary~\ref{cor:agsd-euqe}.
\begin{theorem}\label{th:pGSD-pIUC}
  Let $M$ be a locally compact, but not compact, Polish space.  Assume that \eqref{A0}--\eqref{A4} hold.  
	Then the following conditions are equivalent:
  \begin{enumerate}
    \item\label{th:pGSD-pIUC-a}
        Both semigroups $\left\{U_t: t \geq 0 \right\}$ and $\left\{U^*_t: t \geq 0 \right\}$ are pGSD.

    \item\label{th:pGSD-pIUC-b}
        There exist probability measures $\rho, \rho^* \in \cM_1(M)$ such that $\supp \rho \cap (M \setminus \cut) \neq \emptyset$, $\supp \rho^* \cap (M \setminus \cut) \neq \emptyset$, and an exhausting family of sets $\left\{K_t : t \geq t_0\right\}$ in $M$ such that
        \begin{gather*}
            \sup_{x \in K_t \setminus \cut} \; \sup_{\substack{f \in L^{\infty}(M,\mu)\\ \left\|f\right\|_{\infty} \leq 1}} \left|\frac{U_t f(x)}{U_t \I_M(x)} - \rho(f)\right| \to 0
            \quad\text{as}\quad t \to \infty,
        \intertext{and}
            \sup_{y \in K_t \setminus \cut} \; \sup_{\substack{f \in L^{\infty}(M,\mu)\\ \left\|f\right\|_{\infty} \leq 1}} \left|\frac{U^*_t f(y)}{U^*_t \I_M(y)} - \rho^*(f)\right| \to 0
            \quad\text{as}\quad t \to \infty.
        \end{gather*}
    \end{enumerate}
\end{theorem}

\begin{proof}
The implication \enquote{\ref{th:pGSD-pIUC-a}$\Rightarrow$\ref{th:pGSD-pIUC-b}} follows directly from the previous theorem. In order to show the converse direction, we check that the semigroup $\left\{U_t: t \geq 0 \right\}$ is pGSD. The argument for $\left\{U^*_t: t \geq 0 \right\}$ is similar.

Since $\supp \rho \cap (M \setminus \cut) \neq \emptyset$, we have $\rho(\phi_0) >0$. Recall that $\phi_0 \in L^{\infty}(M,\mu)$. Because of~\ref{th:pGSD-pIUC-b} there exists some $t_1 >t_0$ such that
\begin{gather*}
    \left|\rho(\phi_0)-\frac{U_t \phi_0(x)}{U_t \I_M(x)}\right|
    \leq \frac 12 \rho(\phi_0), \quad t \geq t_1, \;x \in K_t.
\end{gather*}
Hence, we have
\begin{gather*}
    U_t \I_M(x) \leq \frac{2}{\rho(\phi_0)} U_t \phi_0(x) = \frac{2}{\rho(\phi_0)} e^{-\lambda_0 t} \phi_0(x), \quad t \geq t_1, \;x \in K_t.
\end{gather*}
This gives pGSD for the semigroup $\left\{U_t: t \geq 0 \right\}$, and the proof is complete.
\end{proof}

We already know that the finiteness of the heat content for large $t$ implies the uniqueness of the quasi-stationary probability measures for both semigroups $\left\{U_t: t \geq 0 \right\}$ and $\left\{U^*_t: t \geq 0 \right\}$, see Thoeorem~\ref{th:hcf_and_eqe}.\ref{th:hcf_and_eqe-3}. We will now show a much stronger result, which says that the pGSD property and some additional information on the operator norm of $U_t$ and $U_t^*$ for large times, also implies the uniqueness of the quasi-stationary measures. The additional condition can be expressed as follows: there exists $t_2>0$ such that
\begin{align} \label{eq:add_ass_for_uniq}
    \sup_{t \geq t_2} \left(e^{\lambda_0 t}\sup_{x \in M}\big(U_t \I(x) + U^*_t \I(x)\big)\right) < \infty.
\end{align}
Since
\begin{gather*}
\sup_{x \in M} U_t \I(x)
    = \left\|U_t\right\|_{\infty,\infty}
    \quad \text{and} \quad
    \sup_{x \in M} U^*_t \I(x) = \left\|U^*_t\right\|_{\infty,\infty},
\end{gather*}
\eqref{eq:add_ass_for_uniq} is equivalent to the combination of the following two conditions \eqref{eq:equiv_1} and \eqref{eq:equiv_2}:
\begin{align} \label{eq:equiv_1}
    \sup_{t \geq t_2} \frac{\left\|U_t\right\|_{\infty,\infty}}{\left\|U_t\right\|_{2,2}} < \infty,
    \quad
    \sup_{t \geq t_2} \frac{\left\|U^*_t\right\|_{\infty,\infty}}{\left\|U^*_t\right\|_{2,2}} < \infty;
\end{align}
\begin{align} \label{eq:equiv_2}
    \sup_{t \geq t_2} e^{\lambda_0 t} \left\|U_t\right\|_{2,2} < \infty \quad \text{(or, equivalently, \ \ $\sup_{t \geq t_2} e^{\lambda_0 t} \left\|U_t^*\right\|_{2,2} < \infty$)}.
\end{align}
If the semigroups $\left\{U_t: t \geq 0 \right\}$ and $\left\{U^*_t: t \geq 0 \right\}$ are pGSD for the exhausting family of compact sets $\left\{K_t : t \geq t_2\right\}$ the condition \eqref{eq:add_ass_for_uniq} is equivalent to
\begin{gather} \label{eq:equiv_to_add_ass_for_uniq}
    \sup_{t \geq t_2} \left(e^{\lambda_0 t}\sup_{x \in K_t^c}\big(U_t \I(x) + U^*_t \I(x)\big)\right) < \infty.
\end{gather}
The next theorem shows that the uniqueness of the quasi-stationary measures extends to a large class of compact semigroups with infinite heat content.

\begin{theorem} \label{th:pGSD-gives-uniq}
    Let $M$ be a locally compact, but not compact, Polish space. Assume that  \eqref{A0}--\eqref{A4} hold and that both semigroups $\left\{U_t: t \geq 0 \right\}$ and $\left\{U^*_t: t \geq 0 \right\}$ are pGSD 
    and satisfy \eqref{eq:add_ass_for_uniq}. Then the measures $m$ and $m^*$ defined in \eqref{eq:quasi_stat_meas} are the unique quasi-stationary probability measures
		of $\left\{U_t: t \geq 0 \right\}$ and $\left\{U^*_t: t \geq 0 \right\}$, respectively, such that $m(M\setminus\cut)\cdot m^*(M\setminus\cut)>0$.
\end{theorem}

\begin{proof}
We give a proof for the measure $m$ only; the argument for $m^*$ is similar. Suppose that $\bar{m}$ is a quasi-stationary probability measure of $\left\{U_t: t \geq 0 \right\}$ such that $\supp \bar m \cap (M \setminus \cut) \neq \emptyset$.
Following the argument from the proof of Theorem~\ref{th:hcf_and_eqe}.\ref{th:hcf_and_eqe-3}, it is enough to show that $m(f) = \bar{m}(f)$ for $f \in L^{\infty}(M,\mu)$.

For a compact set $K \subset M$ such that $\bar{m}(\I_{K}\phi_0) > 0$ (in particular, $\bar{m}(K) > 0$)  we define
\begin{gather*}
    \sigma_K(dx)
    = \frac{\I_K(x)\bar{m}(dx)}{\bar{m}(K)}.
\end{gather*}
Fix $f \in L^{\infty}(M,\mu)$. We have
\begin{gather*}
    |m(f) - \bar{m}(f)|
    \leq \left|m(f) - \frac{\sigma_K(U_t f)}{\sigma_K(U_t \I_M)}\right| +  \left|\frac{\sigma_K(U_t f)}{\sigma_K(U_t \I_M)} - \bar{m}(f)\right|.
\end{gather*}
Using the convergence assertion from Theorem~\ref{th:pGSD-to-puqe} with the probability measure $\sigma = \sigma_K$, we see that the first term on the right-hand side vanishes as $t \to \infty$. We will now show that for every $\epsilon \in (0,1)$ there exists a compact set $K \subset M$ such that the second term is less than $\epsilon$, uniformly for large $t$.

Fix $\epsilon \in (0,1)$. By \eqref{eq:add_ass_for_uniq},
\begin{gather*}
    c := \sup_{t \geq t_2}\left( e^{\lambda_0 t}  \sup_{x \in M} U_t \I(x)\right) < \infty.
\end{gather*}
Since $\bar{m}$ is a finite measure, there exists a compact set $K = K(\epsilon) \subset M$ as above such that
\begin{align} \label{eq:tightness}
    \bar{m}(K^c)
    \leq \frac{\bar{m}(\phi_0)}{c\left\|\phi_0\right\|_{\infty}(1+\bar{m}(f)+ \left\|f\right\|_{\infty})} \frac{\epsilon}{1+\epsilon}.
\end{align}
The existence of such a set is a consequence of the $\sigma$-compactness of the space $M$.
By the quasi-stationarity property of $\bar{m}$, cf.\ \eqref{eq:quasi-stat-def}, we have
\begin{gather*}
    \bar{m}(\I_{K} U_{t} f) + \bar{m}(\I_{K^c} U_{t} f) = \bar{m}(f) \big(\bar{m}(\I_{K} U_{t} \I_M) + \bar{m}(\I_{K^c} U_{t} \I_M) \big),
\end{gather*}
which implies that
\begin{align*}
    \left|\frac{\sigma_{K}(U_{t} f)}{\sigma_{K}(U_{t} \I_M)} - \bar{m}(f)\right|
    &= \left|\frac{\bar{m}(\I_{K} U_{t} f)}{\bar{m}(\I_{K} U_{t} \I_M)} - \bar{m}(f)\right|\\
    &= \left|\frac{\bar{m}(f) \big(\bar{m}(\I_{K} U_{t} \I_M) + \bar{m}(\I_{K^c} U_{t} \I_M) \big) - \bar{m}(\I_{K^c} U_{t} f)}{\bar{m}(\I_{K} U_{t} \I_M)} - \bar{m}(f)\right|\\
	&= \left|\frac{\bar{m}(f)\bar{m}(\I_{K^c} U_{t} \I_M)  - \bar{m}(\I_{K^c} U_{t} f)}{\bar{m}(\I_{K} U_{t} \I_M)}\right|\\
    &\leq \left(\bar{m}(f)+ \left\|f\right\|_{\infty}\right) \frac{\bar{m}(\I_{K^c} U_{t} \I_M)}{\bar{m}( U_{t} \I_M) - \bar{m}(\I_{K^c} U_{t} \I_M)} \\
    &= \left(\bar{m}(f)+ \left\|f\right\|_{\infty}\right) \frac{\frac{\bar{m}(\I_{K^c} U_{t} \I_M)}{\bar{m}( U_{t} \I_M)}}{\;\;1 - \frac{\bar{m}(\I_{K^c} U_{t} \I_M)}{\bar{m}( U_{t} \I_M)}\;\;}.
\end{align*}
From \eqref{eq:by_phi} we infer that
\begin{gather*}
    \bar{m}( U_{t} \I_M) \geq e^{-\lambda_0 t} \frac{\bar{m}(\phi_0)}{\left\|\phi_0\right\|_{\infty}}, \quad t >0.
\end{gather*}
Combining this inequality with the pGSD property and \eqref{eq:tightness} yields
\begin{gather*}
    \frac{\bar{m}(\I_{K^c} U_{t} \I_M)}{\bar{m}( U_{t} \I_M)}
    \leq c \,\frac{\left\|\phi_0\right\|_{\infty}}{\bar{m}(\phi_0)}\, \bar{m}(K^c)
    \leq \frac{1}{1+\bar{m}(f)+ \left\|f\right\|_{\infty}} \frac{\epsilon}{1+\epsilon},
    \quad t \geq t_2.
\end{gather*}
Since the function $(0,1) \ni x \mapsto x / (1-x)$ is increasing, we finally get
\begin{gather*}
    |m(f) - \bar{m}(f)|
    \leq \limsup_{t \to \infty} \left|\frac{\sigma_{K}(U_{t} f)}{\sigma_{K}(U_{t} \I_M)} - \bar{m}(f)\right| \leq \epsilon,
\end{gather*}
which completes the proof.
\end{proof}

\section{Large time asymptotic behaviour of compact semigroups}\label{sec-large-time}

As a by-product of our investigations on quasi-ergodicity, we also obtain some statements on the asymptotic behaviour of compact $L^2$-semigroups on $L^p$ spaces. In this generality such results seem to be new.

\begin{corollary}\label{cor:hcf_and_asymp}
    Let $M$ be a locally compact Polish space and assume that \eqref{A0}--\eqref{A3} hold. 
	If there exists some $t_1 \geq t_0$ such that the heat content is finite, i.e.\
\begin{gather*}
    Z(t_1) < \infty,
\end{gather*}
then we have the following:
\begin{enumerate}
\item\label{cor:hcf_and_asymp-1}
    For every $p \in [1,\infty]$ there exist $C>0$ such that for every finite measure $\sigma$ on $M$, and for all $t>6t_1$ and $t_1 \leq s \leq t/2$ we have
    \begin{gather*}
        \sup_{\substack{f \in  L^p(M,\mu)\\ \left\|f\right\|_p \leq 1}} \left|e^{\lambda_0 t}\sigma(U_t f) - \frac 1\Lambda \, \sigma(\phi_0) \int_M f(y) \psi_0(y) \,\mu(dy)\right|
        \leq C  \, e^{-\gamma (t-s)} \, e^{\lambda_0 s}\sigma(U_s \I_M),
    \intertext{and}
        \sup_{\substack{f \in  L^p(M,\mu)\\ \left\|f\right\|_p \leq 1}} \left|e^{\lambda_0 t}\sigma(U^*_t f) - \frac 1\Lambda \, \sigma(\psi_0) \int_M f(x) \phi_0(x) \,\mu(dx)\right|
        \leq C  \, e^{-\gamma (t-s)} \, e^{\lambda_0 s}\sigma(U^*_s \I_M).
\end{gather*}
In particular, taking $\sigma = \delta_z$, $z \in M$, and $s = t_1$, we obtain uniform estimates
\begin{gather*}
        \sup_{x \in M} \sup_{\substack{f \in  L^p(M,\mu)\\ \left\|f\right\|_p \leq 1}} \left|e^{\lambda_0 t}U_t f(x) - \frac 1\Lambda \, \phi_0(x) \int_M f(y) \psi_0(y) \,\mu(dy)\right|
        \leq \widetilde C  \, e^{-\gamma t},
    \intertext{and}
        \sup_{y \in M} \sup_{\substack{f \in  L^p(M,\mu)\\ \left\|f\right\|_p \leq 1}} \left|e^{\lambda_0 t}U^*_t f(y) - \frac 1\Lambda \, \psi_0(y) \int_M f(x) \phi_0(x) \,\mu(dx)\right|
        \leq \widetilde C  \, e^{-\gamma t}.
\end{gather*}
\item\label{cor:hcf_and_asymp-2}
For every $p, r \in [1,\infty]$ there exist $C>0$ such that for all $t>6t_1$ we have
    \begin{gather*}
        \sup_{\substack{f \in  L^p(M,\mu)\\ \left\|f\right\|_p \leq 1}} \left\|e^{\lambda_0 t} U_t f - \frac 1\Lambda \, \phi_0 \int_M f(y) \psi_0(y) \,\mu(dy)\right\|_r
        \leq C  \, e^{-\gamma t},
    \intertext{and}
        \sup_{\substack{f \in  L^p(M,\mu)\\ \left\|f\right\|_p \leq 1}} \left\|e^{\lambda_0 t}U^*_t f - \frac 1\Lambda \, \psi_0 \int_M f(x) \phi_0(x) \,\mu(dx)\right\|_r
        \leq C  \, e^{-\gamma t}.
\end{gather*}
\end{enumerate}
\end{corollary}
The above bounds follow from direct inspection of the estimates in the proof of Theorem \ref{th:hcf_and_eqe}. The cases $p=1$ in Part~\ref{cor:hcf_and_asymp-1} and $p=1$, $r=\infty$ in in Part~\ref{cor:hcf_and_asymp-2} are special -- they follow directly from \eqref{A2} and the finiteness of the heat content is actually needed in that case.

Let $\left\{K_t : t \geq t_0\right\}$ be an exhausting family of sets in $M$. By considering the sequence of measures $(\sigma_n)$ such that $\sigma_n(\, \cdot \,) := \mu(\, \cdot \, \cap K_n)$, $n \in \N$, and using Corollary \ref{cor:hcf_and_asymp}.\ref{cor:hcf_and_asymp-1} we obtain the following asymptotics of the heat content.

\begin{corollary}\label{cor:hc_asym}
    Let $M$ be a locally compact Polish space and assume that \eqref{A0}--\eqref{A3} hold.
	If there exists some $t_1 \geq t_0$ such that the heat content is finite, i.e.\
\begin{gather*}
    Z(t_1) < \infty,
\end{gather*}
then there is some $C>0$ such that
\begin{gather*}
\left|e^{\lambda_0 t} Z(t) - \frac{\left\|\phi_0\right\|_1 \left\|\psi_0\right\|_1}{\Lambda}\right| \leq C e^{-\gamma t}, \quad t > 6t_1.
\end{gather*}
\end{corollary}

If $Z(t) = \infty$ for large $t$'s, we cannot expect the asymptotic estimates to be uniform as in Corollary~\ref{cor:hcf_and_asymp}.\ref{cor:hcf_and_asymp-1}. Still, we can prove a progressive version which is similar to Theorem \ref{th:pGSD-to-puqe}.
The following result can be obtained directly by inspection of the proof of that theorem.

\begin{corollary} \label{cor:pGSD-to-asymp}
    Let $M$ be a locally compact, but not compact, Polish space. Assume that \eqref{A0}--\eqref{A4} hold 
    and that both semigroups $\left\{U_t: t \geq 0 \right\}$ and $\left\{U^*_t: t \geq 0 \right\}$ are pGSD for the exhausting family of compact sets $\left\{K_t : t \geq t_0\right\}$. Then for every $a, b \in (0,1)$ such that $a+2b=1$ and for every $p \in [1,\infty]$ there exists some $C>0$ such that for all $t>(4/(a \wedge b)) t_0$
    \begin{gather*}
            \sup_{\sigma \in \cM^1(K_{a t})} \sup_{f \in  L^p(M,\mu) \atop \left\|f\right\|_p \leq 1}
    \left|e^{\lambda_0 t}\sigma(U_t f) - \frac 1\Lambda \, \sigma(\phi_0) \int_M f(y) \psi_0(y) \,\mu(dy) \right|
    \leq C  \kappa_b(t),
    \intertext{and}
        \sup_{\sigma \in \cM^1(K_{a t})} \sup_{f \in  L^p(M,\mu) \atop \left\|f\right\|_p \leq 1}
    \left|e^{\lambda_0 t}\sigma(U^*_t f) - \frac 1\Lambda \, \sigma(\psi_0) \int_M f(x) \phi_0(x) \,\mu(dx) \right|
    \leq C  \kappa_b(t),
    \end{gather*}
		where the time-rate $\kappa_b(t)$ is given by \eqref{eq:kappa-rate}. By considering $\sigma = \delta_z$, $z \in K_{a t}$, we obtain progressive uniform estimate on the sets $K_{a t}$ when $t \to \infty$.
\end{corollary}

\section{Examples and applications} \label{sec:examples}

\subsection{Feynman--Kac semigroups of Feller processes} \label{sec:F-K-general}
Suppose that $(M,d)$ is a locally compact Polish space, and $\mu$ is a positive, locally finite measure with full topological support on the Borel sets $\cB(M)$. We assume that the space $(M,d)$ is \emph{unbounded} (in particular, it is \emph{not compact}) and the measure $\mu$ is \emph{not finite}. Let $(X_t)_{t \geq 0}$ and $(\widehat{X}_t)_{t \geq 0}$ be Markov processes with values in $M$ and the following transition semigroups:
\begin{gather*}
    P_t f(x) = \ex^x f(X_t)
    \quad\text{and}\quad
    \widehat P_t f(y) = \widehat{\ex}^y f(\widehat{X}_t),
    \quad f \in L^2(M,\mu);
\end{gather*}
Throughout this section we assume that $P_t$ and $\widehat P_t$ are contractions on $L^2(M, \mu)$ and that the semigroups $\{P_t:t \geq 0\}$, $\{\widehat P_t :t \geq 0\}$ have both the Feller and the strong Feller property, i.e.\ they are \enquote{doubly Feller} in the sense of Chung \cite{bib:Ch1986}. Moreover, we assume that the corresponding Markov resolvents are in a weak duality relation with respect to the measure $\mu$, see \cite[Definition 13.1]{ChuW2005}. By \cite[Proposition 13.6]{ChuW2005} we have
\begin{gather*}
    \int f(x) \cdot P_t g(x)\,\mu(dx) = \int \widehat P_t f(x) \cdot g(x)\,\mu(dx), \quad f,g \in L^2(M,\mu),\; t>0,
\end{gather*}
this means that $\widehat P_t$ is the (functional analytic) adjoint operator $P_t^*$ of $P_t$ in the space $L^2(M,\mu)$. Finally,
we assume that for every $t>0$ the operators $P_{t}$ and $\widehat P_{t}$ are positivity improving and bounded from $L^2(M,\mu)$ to $L^{\infty}(M,\mu)$.

In particular, the \enquote{free} semigropus $\{P_t, t\geq 0\}$ and  $\{\widehat P_t, t\geq 0\}$ are such that they satisfy our assumptions \eqref{A1}--\eqref{A3}. We are now going to give sufficient criteria such that the corresponding Feynman--Kac semigroups $\{U_t, t\geq 0\}$ and  $\{U_t^*, t\geq 0\}$ satisfy \eqref{A0}--\eqref{A4}.

Let $V$ be a locally bounded potential on $M$ which is bounded below, that is $V = V_{+}-V_{-}$ with $V_{+} \in L^{\infty}_{\loc}(M,\mu)$ and $V_{-} \in L^{\infty}(M,\mu)$, for which we define the \emph{Feynman--Kac semigroup}
\begin{gather*}
    U_t f(x) = \ex^x \left[e^{-\int_0^t V(X_s)\, ds} f(X_t) \right], \quad f \in L^2(M,\mu), \; t>0.
\end{gather*}
By \cite[Theorem 13.25]{ChuW2005} the weak dual or adjoint semigroup consists of operators
\begin{gather*}
    \widehat U_t f(x) =  U^*_t f(x) = \widehat \ex^x \left[e^{-\int_0^t V(\widehat X_s)\,ds} f(\widehat X_t) \right], \quad f \in L^2(M,\mu), \; t>0.
\end{gather*}

We first verify the assumptions \eqref{A1}--\eqref{A3} and show that $U_t$ and $U^*_t$ are integral operators. If $(X_t)_{t\geq 0}$ is a stochastic process, we denote by $\tau_B := \inf\left\{t>0 \,:\, X_t\in B^c\right\}$ the first exit time from the set $B$.
\begin{lemma}\label{lem:general-Feller-ass}
  The following assertions hold.
  \begin{enumerate}
    \item\label{lem:general-Feller-ass-a}
        The operators $U_{t}$, $t>0$, are positivity improving.

    \item\label{lem:general-Feller-ass-b}
		Both semigroups $\left\{U_t: t \geq 0 \right\}$ and $\left\{U^*_t: t \geq 0 \right\}$ are Feller and strong Feller.

		\item\label{lem:general-Feller-ass-c}
    For every $t>0$ the operators $U_t, U^*_t:L^2(M,\mu) \to L^{\infty}(M,\mu)$ are bounded. In particular, for every $t>0$ there exists a measurable kernel $u_t(x,y)$ such that \eqref{eq:kernel} holds, and the semigroups $\{U_t:t\geq 0\}$ and $\{\widehat U_t:t\geq 0\}$ are in duality with respect to the measure $\mu$.
    \end{enumerate}
\end{lemma}

\begin{proof}
\ref{lem:general-Feller-ass-a} Let $0\not\equiv f \in L^2(M,\mu)$ be positive and $t, r >0$. We have
\begin{gather*}
    U_t f(x) = \ex^x\left[e^{-\int_0^t V(X_s) \,ds} f(X_t) \right] \geq e^{-t \left\|\I_{B_r(x)}V_{+}\right\|_{\infty}} \ex^x \left[f(X_t); t < \tau_{B_r(x)}\right].
\end{gather*}
The expectation appearing on the right-hand side tends to $\ex^x f(X_t) = P_tf(x)$ as $r \to \infty$. Since $P_t$ is positivity improving, and $V$ is locally bounded, the expression on the right-hand side is strictly positive for $\mu$-almost every $x$ and sufficiently large $r>0$. To see this, fix any $x$ such that $P_t f(x)>0$. Since $\ex^x \left[f(X_t); t < \tau_{B_r(x)}\right]$ increases to $P_t f(x) = \Ee\left[f(X_t)\right]$ as $r\uparrow\infty$, there is some $r=r(x)$ such that $\ex^x \left[f(X_t); t < \tau_{B_r(x)}\right]>0$, hence $U_t f(x)>0$. This gives the assertion.

Part~\ref{lem:general-Feller-ass-b} follows directly from Chung \cite[Theorem 2]{bib:Ch1986}. The assumptions (a)--(c) of Chung's theorem can be easily verified as $V$ is both bounded below and locally bounded.

We still have to prove~\ref{lem:general-Feller-ass-c}. First, observe that for every $t>0$ the operators $U_t, U^*_t: L^2(M,\mu) \to L^{\infty}(M,\mu)$ are bounded. Indeed, for $f \in L^2(M,\mu)$ and $t>0$ we have
\begin{gather*}
    U_t f(x)
    = \ex^x \left[e^{-\int_0^t V(X_s) ds} f(X_t) \right]
    \leq e^{t \left\|V_{-}\right\|_{\infty}}  \ex^x f(X_t)
    = e^{t \left\|V_{-}\right\|_{\infty}} P_t f(x),
\end{gather*}
implying that $\left\|U_t\right\|_{2,\infty} \leq e^{t \left\|V_{-}\right\|_{\infty}} \left\|P_t\right\|_{2,\infty}$.
In the same way we get $\left\|U^*_t\right\|_{2,\infty} \leq e^{t \left\|V_{-}\right\|_{\infty}} \left\|P^*_t\right\|_{2,\infty}$. Combining this with part
~\ref{lem:general-Feller-ass-b} and the semigroup property, we see that for every $t>0$ the operators $U_t, U^*_t$ map $L^2(M,\mu)$ $\left\{P_t, t\geq 0\right\}$,  $\left\{\widehat P_t, t\geq 0\right\}$ into $L^2(M,\mu) \cap C_b(M)$, and we can apply \cite[Theorem 1.6]{JSch2006} to show that $U_t, U^*_t$, $t>0$, are integral operators. That is, there exist non-negative measurable kernels $u_t(x,y)$ and $u_t^*(x,y)$ such that
\begin{gather*}
    U_t f(x) = \int_M f(y) u_t(x,y) \mu(dy), \quad U^*_t f(x) = \int_M f(y) u^*_t(x,y) \mu(dy), \quad f \in L^2(M,\mu), \; t >0.
\end{gather*}
Since the operators $U_t$ and $U^*_t$ are adjoint, we have $u^*_t(x,y) = u_t(y,x)$.
\end{proof}

Now we give general sufficient conditions for the compactness of the operators $U_t$, $U^*_t$, $t>0$ (cf.\ assumption \eqref{A0}) and for the finiteness of the heat content for large times. Recall that the heat content is defined as
\begin{gather*}
    Z(t) = \int_M U_t \I(x) \mu(dx) = \int_M U^*_t \I(y) \mu(dy), \quad t>0.
\end{gather*}

\begin{lemma} \label{lem:general-heat-content} For $x \in M$ and $t>0$, we have
	\begin{align} \label{eq:mass_est}
        e^{-t \cdot \ess_{y}\left(\I_{B_1(x)}(y) V_+(y)\right)} \, \pr^x(t \leq \tau_{B_1(x)}) \leq U_t \I(x) \leq \frac{1}{t} \int_0^t P_s (e^{-tV(\cdot)})(x) ds
	\end{align}
    \textup{(}an obvious modification gives similar estimates for $U^*_t \I(x)$\textup{)}. In particular,
    \begin{align} \label{eq:hc_est}
        \inf_{x \in M} \pr^x(t \leq \tau_{B_1(x)}) \int_M e^{-t \cdot \ess_{y}\left(\I_{B_1(x)}(y) V_+(y)\right)} \mu(dx) \leq Z(t) \leq \int_M e^{-tV(x)} \mu(dx). 
	\end{align}
    for $t>0$. Consequently, the following assertions hold.
    \begin{enumerate}
    \item\label{lem:general-heat-content-a}
     If
     \begin{gather} \label{eq:conf}
        \lim_{d(x,x_0) \to \infty} V(x) = \infty,
     \end{gather}
     then the operators $U_t$, $U^*_t$, $t>0$, are compact; in particular, \eqref{A0} holds.
		
    \item\label{lem:general-heat-content-b}
		Fix $t>0$. Then:
    \begin{gather*}
        \int_{M} e^{-t V(x)}\mu(dx) < \infty \implies Z(t)<\infty.
    \end{gather*}
    On the other hand, if $\inf_{x \in M} \pr^x(t \leq \tau_{B_1(x)}) >0$, then:
    \begin{gather*}
        Z(t)<\infty \implies \int_M e^{-t \cdot \ess_{y}\left(\I_{B_1(x)}(y) V_+(y)\right)} \mu(dx) < \infty.
    \end{gather*}
		\end{enumerate}
\end{lemma}

\begin{proof}
Fix $t>0$. For the proof of the upper bound in \eqref{eq:mass_est} it is enough to use Jensen's inequality and Tonelli's theorem
\begin{gather*}
    U_t \I(x)
    = \ex^x\left[e^{-\int_0^tV(X_s)\,ds}\right]
    \leq \ex^x\left[\frac{1}{t} \int_0^t e^{-t V(X_s)}\,ds\right]
    = \frac{1}{t} \int_0^t P_s (e^{-tV(\cdot)})(x) \, ds
\end{gather*}
and the upper estimate in \eqref{eq:hc_est} follows by integration with respect to $\mu(dx)$ and duality. The lower bounds in \eqref{eq:mass_est} and \eqref{eq:hc_est} follow from the definitions.
Observe that Part~\ref{lem:general-heat-content-b} follows directly from the latter inequality.

To get Part~\ref{lem:general-heat-content-a}, we show \eqref{A4} for every $t_0>0$ and then we use Lemma \ref{lem:a4_gives_a0}. By \eqref{eq:conf} there exists a function $g \in C_{\infty}(M)$ such that
\begin{align} \label{eq:cpt_aux_1}
    e^{-tV(x)} \leq g(x),\quad\text{for $\mu$ almost every $x\in M$.}\footnotemark
\end{align}
\footnotetext{Indeed: fix an arbitrary $x_0 \in M$, set $h(r) := \essi_{\left\{x \in M: \, d(x,x_0) \geq r \vee 0\right\}} V_+(x)$ for $r \in \R$, and use a continuous function $\phi\geq 0$ with support in $[0,1]$ and $\int_0^1 \phi(s) \, ds = 1$ to get $h_{\phi}(r) := h*\phi(r) \leq h(r+1)$, $r \in \R$. Then $V_+(x) \geq h(d(x,x_0)) \geq h_{\phi}(d(x,x_0)-1)$, $x \in M$. Since $t$ is fixed, $g(x) := \exp \big(t \big(\left\|V_-\right\|_{\infty} - h_{\phi}(d(x,x_0)-1)\big)\big)$ is a continuous function which vanishes at infinity.}
Hence, by \eqref{eq:mass_est} and \eqref{eq:cpt_aux_1},
\begin{gather*}
0 \leq U_t \I(x) \leq \frac{1}{t} \int_0^t P_s g (x) ds, \quad x \in M.
\end{gather*}
By the Feller property, the function $x \mapsto (1/t)\int_0^t P_s g (x)$ ds belongs to $C_{\infty}(M)$, which implies that \eqref{A4} holds for any $t\geq t_0$ and any $t_0>0$.

The proof of \eqref{A4} for adjoint operators $U_t^*$ is similar.
\end{proof}
In particular, Lemma~\ref{lem:general-heat-content}.\ref{lem:general-heat-content-b} shows that the finiteness of the heat content is completely independent of the free process.

\bigskip

We will now discuss some concrete classes of Feller processes which are covered by the framework described above. In order to keep the exposition short, we give in each case standard references for the notation and further results.

\subsubsection{L\'evy processes} \label{sec:Levy}
Standard references: Sato~\cite{bib:Sat}, Jacob~\cite[Vol.~1]{bib:J}, Schilling~\cite{bib:barca}.
Throughout this section, the reference measure $\mu$ is $d$-dimensional Lebesgue measure and $d(\cdot,\cdot)$ is the Euclidean distance.

A \emph{L\'evy process} $(X_t)_{t\geq 0}$ is a stochastic process on $\real^d$ with c\`adl\`ag (right-continuous with finite left limits) paths and stationary and independent increments. A L\'evy process can be described in terms of its characteristic exponent $\psi(\xi) = -\log \Ee e^{i\xi\cdot X_1}$. It is known that $\psi$ is uniquely determined by the L\'evy--Khintchine formula
\begin{gather}\label{eq:lkf}
  \psi(\xi)
  =  - i \xi \cdot b + \frac 12 \xi \cdot Q \xi  + \int_{\real^d \setminus \left\{0\right\}} \left(1 - e^{i \xi \cdot y} + i \xi \cdot y \I_{(0,1)}(|y|)\right)\nu(dy) , \quad \xi \in \real^d,
\end{gather}
or by the corresponding the L\'evy triplet $(b,Q,\nu)$, which is given by $b \in \real^d$ (\emph{drift term}), $Q$ is a symmetric, positive semidefinite $d \times d$ matrix (\emph{Gaussian covariance matrix}) and $\nu$ is a measure on $\real^d \setminus \left\{0\right\}$ such that $\int_{\real^d \setminus \left\{0\right\}} (1 \wedge |x|^2) \nu(dx) < \infty$ (\emph{L\'evy measure}).

L\'evy processes are Markov processes, which are invariant under translations, i.e.\ the transition semigroup is a convolution semigroup of the form
\begin{gather*}
    P_t f(x) = \Ee^x \left[f(X_t)\right] = \Ee^0\left[f(X_t+x)\right] = f*\widetilde\mu_t(x)
\end{gather*}
where $\widetilde\mu_t(dy) = \Pp^0(X_t \in -dy)$ is the law of $-X_t$. The dual process $(\widehat{X}_t)_{t \geq 0}$ of a L\'evy process  $(X_t)_{t \geq 0}$ is given by $\widehat{X}_t = - X_t$, and its transition semigroup $\big\{ \widehat P_t:t \geq 0\big\}$ is defined accordingly. The convolution structure of $P_t$ shows that $(P_t)_{t\geq 0}$ is both a Feller semigroup on $C_\infty(\real^d)$ and an $L^p(\real^d,dx)$ Markov semigroup for any $p\in [1,\infty)$.

The process $(X_t)_{t \geq 0}$ is a strong Feller process, i.e.\ $P_t$ maps $L^{\infty}(\real^d,dx)$ into $C_b(\real^d)$, if and only if $\mu_t(dy) = p_t(y)\,dy$ is absolutely continuous w.r.t.\ Lebesgue measure, cf.\ Jacob~\cite[Vol.~1, Lemma 4.8.19, 4.8.20]{bib:J}; in particular, the dual process is strong Feller, too. A sufficient condition for the existence of the densities $p_t(x)$, $t>0$, is that $Q  \neq 0$ or that $\nu$ is infinite and absolutely continuous with respect to Lebesgue measure, see e.g.~\cite[Theorem~27.7]{bib:Sat}, see also the discussion in \cite{bib:KSch}. If, in addition, there is some $t_0>0$ such that $\int_{\real^d} e^{-t_0 \Re \psi(\xi)}\, d\xi < \infty$, then $\sup_{t \geq t_0} \sup_{x \in \real^d} p_t(x) < \infty$, thanks to the Fourier inversion formula. Consequently, the operators $P_t, \widehat P_t:L^2(\real^d,dx) \to L^{\infty}(\real^d,dx)$ are bounded for every $t \geq t_0$. In general, if the integral $\int_{\real^d} e^{-t_0 \Re \psi(\xi)}\, d\xi$ is finite for every $t_0>0$, then we get the existence of bounded and jointly (in $(t,x)$) continuous densities for every $t>0$, see \cite[Lem.~2.1]{bib:KSch}.

In order to see that $P_t$ and $\widehat P_t$ are positivity improving, it is enough to show that $\pr^0(X_t \in B) = \pr^0(\widehat{X}_t \in -B) = \int_B p_t(x)\,dx >0$ for every Borel set $B \subset \real^d$ with $\mathrm{Leb}(B)>0$. Let us fix such a Borel set $B$.
Using the L\'evy--Khintchine formula, we can decompose a L\'evy process into a sum of independent L\'evy processes $X_t = W_t + Y_t$ whose triplets are $(b,Q,0)$ and $(0,0,\nu)$. If $\mathrm{det}(Q)\neq 0$, then we always have a nondegenerate Gaussian part, and it is clear that
\begin{gather*}
    \Pp^0(X_t\in B) = \int \Pp^0(W_t + y\in B) \,\Pp^0(Y_t\in dy) > 0
\end{gather*}
since $\Pp^0(W_t + y \in B)>0$ for all $y\in\real^d$.

Let us assume that $\mathrm{det}(Q)=0$. If the L\'evy measure $\nu$ satisfies $\nu(dy) \geq h(y)\,dy$ with a density $h(y)$  which is strictly positive in some neighbourhood of $0\in\real^d$, say $B_\epsilon(0)$, then we can decompose $X_t = Z_t + C_t$ into two independent L\'evy processes whose triplets are $(b,Q,\nu(dy)-h(y)\,dy)$ and $(0,0,h(y)\,dy)$.  Since $C_t$ is a compound Poisson process, we know from \cite[Cor.~3.5]{bib:barca} that
\begin{gather*}
    \Pp^0(C_t \in B) = e^{-h t}\delta_0(B) + e^{-h t} \sum_{k=1}^\infty \frac{t^k}{k!} \int_B h^{*k}(x)\,dx,
    \quad
    h = \int_{\real^d} h(x)\,dx.
\end{gather*}
Since the $k$-fold convolution $h^{*k}(x)$ is strictly positive in $B_{\epsilon k}(0)$, we conclude that $\Pp^0(C_t \in B+z)>0$ for any $z\in\real^d$, hence $\Pp^0(X_t\in B) = \int \Pp^0(C_t + z\in B)\,\Pp^0(Z_t\in dz)>0$. With a little more effort, the above argument even shows that the density of $X_t$ satisfies $p_t(x)>0$ for all $t>0$ and Lebesgue a.a.\ $x\in\real^d$.

It remains to check the conditions for compactness and the finiteness of the heat content stated in Lemma~\ref{lem:general-heat-content}. Let $x_0 = 0$.
Choosing the potential $V$ in such a way that $V(x) \to \infty$ as $|x| \to \infty$, we obtain that the Feynman--Kac semigroup operators $U_t, \widehat U_t$, $t>0$, are compact operators on $L^2(M,\mu)$. Moreover, if we know that
\begin{gather*}
    \liminf_{|x| \to \infty} \frac{V(x)}{\log|x|} > 0,
\end{gather*}
then there is $t>0$ such that $Z(t)<\infty$.

\subsubsection{Levy-type processes.} Standard references: Jacob~\cite[Vol.~1]{bib:J}, B\"{o}ttcher \emph{et al.}~\cite{bib:matters}.
A L\'evy-type or Feller process is a Markov process $X_t$ with values in $\real^d$ which behaves locally like a L\'evy process. In particular, the transition semigroup $P_tu(x) = \Ee^x u(X_t)$ is a Feller semigroup and one can show that the infinitesimal generator is a pseudo-differential operator of the form $Au(x) = \mathcal -F^{-1}_{\xi\to x} \left(p(x,\xi)\mathcal Fu(\xi)\right)$ ($\mathcal F$ denotes the Fourier transform, $u$ is a sufficiently regular function, e.g.\ a test function $u\in C_c^\infty(\real^d)$) whose symbol $p(x,\xi)$ is, for fixed $x$, the characteristic exponent of a L\'evy process. Therefore, $p(x,\xi)$ is given by a L\'evy--Khintchine formula \eqref{eq:lkf} where the L\'evy-triplet depends on $x$, i.e.\ $(b,Q,\nu(dy)) = (b(x),Q(x),\nu(x,dy))$. Thus, $A$ is a \enquote{L\'evy generator with variable coefficients}. A typical example are stable-like generators of the form $Au(x) = (-\Delta)^{\alpha(x)}u(x)$ where $\alpha:\real^d \to [0,2]$ is a variable order of differentiability and the symbol is of the form $p(x,\xi) = |\xi|^{\alpha(x)}$

It is a major problem to construct a L\'evy-type process given an admissible symbol $p(x,\xi)$ or, which is the same, an $x$-dependent L\'evy triplet. A thorough discussion is in Jacob~\cite[Vol.~3]{bib:J} and B\"{o}ttcher \emph{et al.}~\cite{bib:matters}, probably the most general results to-date are due to K\"{u}hn~\cite{bib:kuehn} and Knopova \emph{et al.} \cite{bib:KKS}.

In what follows we indicate some conditions, under which the assumptions needed in Section~\ref{sec:F-K-general} are satisfied. Our main intention is not to cover the most general situation, but to illustrate that within the class of L\'evy-type processes there are many examples to be found. The following assumptions on the symbol are frequently used
\begin{align}
\label{ltp-1}
    C_1 |\xi|^{\alpha_0}  &\leq \mathrm{Re} p(x,\xi) &&\forall x\in\real^d,\; |\xi|\geq 1, &&\text{(minimal growth)}\\
\label{ltp-2}
    |p(x,\xi)-p(y,\xi)| &\leq C_2 |x-y|^\gamma |\xi|^\alpha &&\forall x\in\real^d,\; |\xi|\geq 1, &&\text{(H\"{o}lder condition)}\\
\label{ltp-3}
    |p(x,\xi)| &\leq C_3 (1+|x|^\gamma) |\xi|^\beta &&\forall x\in\real^d,\; |\xi|\leq 1, &&\text{(growth of coefficients)}
\end{align}
where $C_1,C_2,C_3$ are constants and $\alpha_0,\alpha,\beta\in (0,2]$ and $\gamma\in (0,\beta)$ are suitable exponents.

Assume that $p(x,\xi)$ is the symbol of a L\'evy-type process. If either $\alpha \geq \gamma > \frac{\alpha}{\alpha_0}\left(\alpha-\alpha_0+\frac 12\right)$ or $\alpha \leq \gamma < \frac 2\alpha_0-\frac 12$, then $X_t$ has a transition density $p_t(x,y)$ such that $\sup_x \int p_t^2(x,y)\,dy<\infty$, i.e.\ \eqref{A3} holds for the free semigroup $P_t$, see \cite[Thm.~2.14]{bib:kuehn}.

The existence of a L\'evy-type process with a given symbol $p(x,\xi)$ is a much more delicate issue. Here, we content ourselves to mentioning that in the stable-like case $p(x,\xi) = |\xi|^{\alpha(x)}$ where $\alpha : \real^d \to [a_1,a_2] \subset (0,2)$ is H\"{o}lder continuous, the existence of a stable-like process with a continuous density $(t,x,y)\mapsto p_t(x,y) > 0$ has been established in \cite[Thm.~3.8, Thm.~3.11]{bib:kuehn} as well as in \cite[Section~3]{bib:KKS}. In particular, we get Feller and strong Feller processes. It is also clear that stable-like symbols satisfy the assumptions \eqref{ltp-1}--\eqref{ltp-3}. We would like to mention, that the methods presented in these papers go way beyond the stable-like case, \cite{bib:kuehn} covers symbols of the type $p(x,\xi) = \psi_{a(x)}(|\xi|)$ where $z\mapsto \psi_a(z)$ has a holomorphic extension on a bow-tie shaped domain of the complex plane (and satisfies certain polynomial boundedness assumptions) while \cite{bib:KKS} are mainly interested in non-rotationally symmetric symbols. Therefore \eqref{A1} and \eqref{A2} are also satisfied.

In general, it is very difficult to determine the symbol of the (formal) adjoint to $A$. Building on \cite{bib:sch-uem}, the paper \cite{bib:sch-wan} contains the general form of the symbol of the adjoint operator for a wide class of symbols (Section 3) and a fully-worked out discussion of the case $p(x,\xi)=|\xi|^{\alpha(x)}$. If $\alpha(x)$ is H\"{o}lder, takes values in $[a_1,a_2]\subset (0,2)$ and satisfies \eqref{ltp-1}--\eqref{ltp-3}, then so does the adjoint symbol $p^*(x,\xi)$ which is given by a L\'evy--Khintchine formula with triplet $(0,0,w(x)|x-y|^{-d-\alpha(x)}\,dy)$ for the weight $w(x) = \alpha(x)2^{\alpha(x)-1}\Gamma\left(\frac 12\alpha(x)+\frac 12d\right)\big/\Gamma\left(1-\frac 12\alpha(x)\right)$. This means that the assumptions \eqref{A1}--\eqref{A3} also hold for the adjoint free semigroup $\widehat P_t$ in this case. Of course, other symbols require case-by-case investigations, e.g.\ based on the above-mentioned literature.

\subsubsection{Continuous-time Markov chains on a countably infinite uniformly discrete metric space.}
Standard references: Murugan \& Saloff--Coste \cite{bib:MS-C1,bib:MS-C2} and Cygan, Kaleta \& \'Sliwi\'nski \cite{bib:CKS}.
Let $(M,d)$ be a countably infinite, unbounded, complete metric space, which is also \emph{uniformly discrete}; this means that there exists a number $a>0$ such for any two distinct points $x,y \in M$ one has $d(x,y) \geq a$, see \cite{bib:MS-C1,bib:MS-C2}. The underlying measure $\mu:M \to (0,\infty)$ is assumed to be a \emph{Frostman measure} with exponent $d_M>0$, i.e.\ there is a constant $c>0$ such that $\mu(B_r(x)) \leq c r^{d_M}$ for any $x \in M$ and $r>0$.

In particular, we have for all $\gamma>d_M$
\begin{align*}
    \int_{B_1(x_0)^c} d(x,x_0)^{-\gamma}\, \mu(dx)
    & = \sum_{n=0}^{\infty} \int_{2^n \leq d(x,x_0) < 2^{n+1}} d(x,x_0)^{-\gamma}\, \mu(dx) \\
    & \leq \sum_{n=0}^{\infty} 2^{-\gamma n} \mu(B_{2^{n+1}}(x_0)) \leq c 2^{d_M}\sum_{n=0}^{\infty} 2^{-(\gamma-d_M) n}
    < \infty.
\end{align*}
Under these assumptions, we can give a sufficient condition for the integrability of the function $x \mapsto e^{-tV(x)}$ which is needed for the finiteness of the heat content, see Lemma~\ref{lem:general-heat-content}.\ref{lem:general-heat-content-b}. Indeed, if there exist $c_0, R>0$ and $x_0 \in M$ such that
\begin{gather*}
    \frac{V(x)}{\log d(x,x_0)} \geq c_0 >0 , \quad d(x,x_0) \geq R,
\end{gather*}
then for every $t > d_M/c_0$ we have
\begin{gather*}
    \int_M e^{-tV(x)} \mu(dx) \leq e^{\left\|V_-\right\|_{\infty}} \mu(B_R(x_0)) + \int_{B_R(x_0)^c} d(x,x_0)^{-c_0t} \mu(dx) < \infty.
\end{gather*}

Consider two probability kernels $Q, \widehat Q:M \times M \to [0,1]$, $\sum_{y \in M} Q(x,y) = 1=\sum_{y \in M} \widehat Q(x,y)$, $x \in M$, which are connected via the duality relation with respect to the measure $\mu$, i.e.\
\begin{align}\label{eq:duality_rel}
    \mu(x) Q(x,y) = \mu(y) \widehat Q(y,x), \quad x,y \in M.
\end{align}
Thus, there are two time-homogeneous Markov chains $\{Z_n: n \in \N_0\}$, $\{\widehat Z_n: n \in \N_0\}$ with values in $M$ and one-step transition probabilities given by $Q$ and $\widehat Q$, respectively. We denote by $\pr^x$, $\widehat \pr^x$ the laws of the chains starting at $x \in M$, i.e.\ $\pr^x (Z_n=y) = \pr(Z_n=y \mid Z_0=x)$ and $\widehat \pr^x (\widehat Z_n=y) = \pr(\widehat Z_n=y \mid \widehat Z_0=x)$. We have $\pr^x (Z_{n}=y) = Q_n(x,y)$ and $\widehat \pr^x (\widehat Z_{n}=y) = \widehat Q_n(x,y)$, where $Q_0(x,y) = \I_{\left\{x\right\}}(y)$, $Q_1(x,y)=Q(x,y)$ and $Q_{n+1}(x,y) = \sum_{z \in M} Q_n(x,z)Q_1(z,y)$, $n \geq 1$, and $\widehat Q_n$'s are defined accordingly. The densities with respect to $\mu$ are given by
\begin{gather*}
    q_n(x,y) = \frac{\pr^x (Z_{n}=y)}{\mu(y)} = \frac{Q_n(x,y)}{\mu(y)}
    \quad \text{and} \quad
    \widehat q_n(x,y) = \frac{\widehat \pr^x (\widehat Z_{n}=y)}{\mu(y)} = \frac{\widehat Q_n(x,y)}{\mu(y)}.
\end{gather*}
Clearly, \eqref{eq:duality_rel} extends to $\mu(x) Q_n(x,y) = \mu(y) \widehat Q_n(y,x)$, $x,y \in M$, $n \in \N$, and we have $\widehat q_n(x,y) = q_n(y,x)$.
The corresponding $n$-step transition operator is
\begin{gather*}
    Q_n f(x) = \sum_{y \in M} Q_n(x,y) f(y) = \sum_{y \in M} q_n(x,y) f(y) \mu(y)
\end{gather*}
whenever the series is finite; $\widehat Q_n f$ is defined in a similar way.

The main objects are the continuous-time chains $\{X_t: t \geq 0\}$ and $\{\widehat X_t: t \geq 0\}$ defined by $X_t:= Z_{N_t}$, $\widehat X_t:= \widehat Z_{N_t}$, respectively, where $\{N_t: t \geq 0\}$
is an independent Poisson process with parameter $\lambda=1$. Clearly, we have $\pr^x(X_t \in A) = \sum_{y \in A} P_t(x,y)$, $A \subset M$, where
\begin{gather*}
    P_t(x,y) = e^{-t}\sum_{n=0}^{\infty} \frac{t^n}{n!} Q_{n}(x,y), \quad t>0, \ x \in M.
\end{gather*}
The transition semigroup $\{P_t: t \geq 0\}$ of this process is defined as
\begin{gather*}
    P_t f(x) = \sum_{y \in M} f(y) P_t(x,y) = \sum_{y \in M} f(y) p_t(x,y) \mu(y), \quad \text{with} \quad p_t(x,y):=\frac{P_t(x,y)}{\mu(y)},
\end{gather*}
for all admissible functions $f$ on $M$. The transition probabilities $\widehat P_t(x,y)$, kernels $\widehat p_t(x,y)$ and the transition semigroup $\{\widehat P_t: t \geq 0\}$ of the process $\{\widehat{X}_t: t \geq 0\}$ are defined accordingly. By the Cauchy--Schwarz inequality we can easily see that all operators $P_t$ and $\widehat{P}_t$ are contractions on $L^2(M,\mu)$; again, we have $\widehat p_t(x,y) = p_t(y,x)$, which means that the weak duality relation, which was described at the beginning of Section \ref{sec:F-K-general}, is in force. Moreover, it is straightforward  to see that both transition semigroups are doubly Feller. In particular, the strong Feller property follows from the fact that $(M,d)$ is uniformly discrete. One can also check that the ultracontractivity of these semigroups is equivalent to the boundedness of the kernel $q_1(x,y)$, e.g.\ if $\inf_{x \in M} \mu(x)>0$. The positivity improving property follows from the strict positivity of the kernels $P_t(x,y)$ and $\widehat P_t(x,y)$. This is the case if the chains $\{Z_n: n \in \N_0\}$, $\{\widehat Z_n: n \in \N_0\}$ are irreducible, i.e.\ for every $x,y \in M$ there is $n=n(x,y) \in \N$ such that $P_n(x,y), \widehat P_n(x,y) >0$.

We can now explicitly estimate the space-rates of the quasi-ergodic behaviour of the Feynman--Kac semigroups with confining potentials (see Corollary \ref{cor:hcf_uqe}). These rates take the form $U_{t_1}\I_M(x)/\phi_0(x)$ and $U^*_{t_1}\I_M(x)/\psi_0(x)$, for some $t_1>0$. The functions $U_{t_1}\I_M(x)$ and $U^*_{t_1}\I_M(x)$ can be estimated by a constant or by using the upper bound in \eqref{eq:mass_est}. Sharper estimates for a large class of processes will be provided in \cite{bib:CKSS}. On the other hand, sharp lower estimates for the ground states for a fairly general class of processes and rather general confining potentials can be found in \cite{bib:CKS}.

If we have good lower estimates for the ground states, then we can apply Lemma \ref{lem:eta_def} and Theorem \ref{th:pGSD} to derive the progressive GSD and progressive exponential quasi-ergodicity from the finiteness of the heat content for large times. Such estimates can, e.g., be found in \cite{bib:CKS}.

\subsubsection{Continuous-time simple random walk on an infinite graph.}
Standard reference: Barlow~\cite{bib:Bar}.
This is a special case of the example in the previous section: Let $M$ be a countably infinite set and  $\Gamma=(M,E)$ be a connected and locally finite graph over $M$. We equip $M$ with its geodesic metric $d: M \times M \to \N_0$; the underlying measure and the transition probability of a simple random walk on the graph is defined in terms of a conductance network on $\Gamma$. The Frostman property of the measure $\mu = (\mu_x)_{x\in M}$ is, for example, guaranteed by the following condition: there is a constant $c_1>0$ such that $\mu_x \leq c_1$, $x \in M$, and the cardinality of a geodesic ball $B_n(x)$ is controlled by a power-type function, i.e.\ $\sup_{x\in M} \sup_{n\in\nat} n^{-d_M}|B_n(x)| <\infty$ for some $d_M>0$.

\subsubsection{Brownian motion on simple nested fractals.}
Standard references: Lindstr\o m \cite{bib:Lin}, Kusuoka \cite{bib:Kus}, Fukushima~\cite{bib:Fuk}, Kumagai \cite{bib:Kum}.
Let $M \subset \R^d$, $d \geq 1$, be an (unbounded) \emph{simple nested fractal} and let $(X_t)_{t \geq 0}$ be a Brownian motion with the state space $M$. We refer the reader to \cite{bib:Lin,bib:Kus,bib:Fuk} for the definition of a simple nested fractal and various constructions of Brownian motion thereon. The set $M$ inherits the Euclidean topology from $\R^d$, and the measure $\mu$ is the Hausdorff measure on $M$ (in particular, $\mu$ is a $d_M$-Frostman measure, where $d_M$ is the Hausdorff dimension of $M$). It is known that $(X_t)_{t \geq 0}$ is a Markov process with symmetric, jointly continuous and bounded transition densities that satisfy sub-Gaussian estimates \cite[Section 5]{bib:Kum}. Consequently, the transition semigroup of the process $(X_t)_{t \geq 0}$ is a doubly Feller, ultracontractive and positivity improving semigroup.

\subsubsection{Sobordinate processes}
Standard references: Sato~\cite{bib:Sat}, Schilling, Song \& Vondra\v{c}ek \cite{bib:SSV}.
Another class of processes which is covered by our results are subordinate processes.
Recall that a subordinator is an increasing L\'evy process with values in $[0,\infty)$, starting at $0$. It is uniquely determined by its Laplace transform: we have $\ex e^{-\lambda S_t} = e^{- t \phi(\lambda)}$, $t>0$, where the exponent $\phi:(0,\infty) \to [0,\infty)$ is a Bernstein function of the form
\begin{gather*}
    \phi(\lambda) = b  \lambda + \int_0^{\infty} (1- e^{-\lambda u}) \,\rho(du).
\end{gather*}
Here $b \geq 0$ is the drift term and $\rho$ is a measure on $(0,\infty)$ such that $\int_0^{\infty} (1 \wedge u) \rho(du) < \infty$ (L\'evy measure).

Let $\{X_t: t \geq 0\}$ and $\{\widehat X_t: t \geq 0\}$ be Markov processes as in Section \ref{sec:F-K-general}, and let $\{S_t: t \geq 0\}$ be a subordinator which is independent of $\{X_t: t \geq 0\}$ and $\{\widehat X_t: t \geq 0\}$. We define the subordinate processes $\{X^{\phi}_t: t \geq 0\}$, $\{\widehat X^{\phi}_t: t \geq 0\}$ by the formulas $X^{\phi}_t:= X_{S_t}$, $\widehat X^{\phi}_t:= \widehat X_{S_t}$, respectively. By independence, their transition semigroups
$\{P^{\phi}_t: t \geq 0\}$ and $\{\widehat P^{\phi}_t: t \geq 0\}$ are given by
\begin{gather*}
    P^{\phi}_t f(x) = \int_{[0,\infty)} P_s f(x) \,\pr(S_t \in ds),
    \quad
    \widehat P^{\phi}_t f(x) = \int_{[0,\infty)} \widehat P_s f(x) \,\pr(S_t \in ds),
\end{gather*}
for $f \in L^2(M,\mu)$ or $f \in L^{\infty}(M,\mu)$. It is obvious that these formulas define contractions on $f \in L^2(M,\mu)$, and the weak duality relation described in Section \ref{sec:F-K-general} holds. Moreover, these semigroups inherit the Feller property, the strong Feller property (provided that $S_t$ is not a compound Poisson subordinator), and the positivity improving property from the semigroups $\{P_t: t \geq 0\}$ and $\{\widehat P_t: t \geq 0\}$. The only regularity property that requires an additional assumption is ultracontractivity: Suppose that $P_t, \widehat P_t$ are bounded operators from $L^2(M,\mu)$ to $L^{\infty}(M,\mu)$ for every $t>0$. By Remark \ref{rem:discuss_ass}, this is equivalent to the boundedness from $L^1(M,\mu)$ to $L^{\infty}(M,\mu)$, for any $t>0$. Therefore, ultracontractivity of the subordinate semigroup and its dual is ensured by the assumption that
\begin{gather*}
    \left\|P^{\phi}_t\right\|_{1,\infty} \leq \int_{[0,\infty)} \left\|P_s\right\|_{1,\infty} \pr(S_t \in ds) < \infty, \quad t>0.
\end{gather*}

\subsection{Progressive uniform exponential quasi-ergodicity for the harmonic oscillator}\label{subsec:HO}

Let $H= -\Delta+|x|^2$ be the Schr\"odinger operator acting in $L^2(\real^d,dx)$, and denote by $U_t = e^{-tH}$, $t \geq 0$, the operators of the corresponding Schr\"odinger semigroup. It is well known that, for every $t>0$, $U_t$ is a compact, self-adjoint and positivity-improving integral operator with the kernel
\begin{gather} \label{eq:HO_kernel}
    u_t(x,y) = (2 \pi \sinh(2t))^{-d/2} \exp\left(-\frac{1}{4}\left(\tanh(t)|x+y|^2+ \coth(t)|x-y|^2\right)\right),
\end{gather}
see e.g.\ \cite[(1.4)]{bib:ST2005}. In particular, the assumptions \eqref{A0}--\eqref{A3} are satisfied. The ground state eigenfunction and eigenvalue are given by
\begin{gather} \label{eq:HO_gs}
    \phi_0(x) = \pi^{-d/4} e^{-\frac{|x|^2}{2}} \quad \text{and} \quad \lambda_0 = d,
\end{gather}
and (see \cite[Proposition 3.3]{bib:ST2005}\footnote{The statement of the result in this reference contains a minor typographical error:\  the factor $2\pi$ is missing before $\cosh(2t)$, the proof is, however, correct})
\begin{gather} \label{eq:HO_sem}
    U_t \I_{\real^d}(x) = (\cosh(2t))^{-d/2} \exp\left(-\frac{|x|^2}{2\coth(2t)}\right),\quad t >0.
\end{gather}
It is easy to check that the semigroup $\left\{U_t: t \geq 0 \right\}$ is not aGSD, cf.\ \cite[Example 4.4]{bib:KKL2018}; in particular, using Corollary~\ref{cor:agsd-euqe}, it is not uniformly quasi-ergodic, see also the discussion in \cite[Section 4.3]{bib:KP}. We will show that our present results can be directly used to describe the true quasi-ergodic properties of this semigroup.

First of all, observe that \eqref{eq:HO_sem} shows that the heat content $Z(t) = \int_{\real^d} U_t \I(x) dx$ is finite for every $t>0$, so that our results from Section~\ref{sec:content} apply. In particular, Theorem~\ref{th:hcf_and_eqe} ensures that the semigroup $\left\{U_t: t \geq 0 \right\}$ is exponentially quasi-ergodic for every finite initial distribution, and the limiting measure
\begin{gather*}
    m(dx) =  \frac{\phi_0(x)}{\left\|\phi_0\right\|_1}\, dx = \frac{1}{(2\pi)^{d/2}} e^{-\frac{|x|^2}{2}}\, dx
\end{gather*}
is the only quasi-stationary measure of $\left\{U_t: t \geq 0 \right\}$. By Corollary~\ref{cor:hcf_uqe} we also get a pointwise version with explicit space-rate
\begin{gather*}
    \frac{U_1\I_{\real^d}(x)}{\phi_0(x)} = \left(\frac{\sqrt{\pi}}{\cosh2}\right)^{d/2} \exp\left(-\frac{|x|^2}{e^{4}-1}\right),
\end{gather*}
which is uniform on compact sets.

We will show that the quasi-ergodicity is progressively uniform. Observe that the semigroup $\big\{U_t:t \geq 0\big\}$ is pGSD, cf.\ Definition~\ref{def:pGSD}.\ref{def:pGSD-b}.

\begin{proposition} \label{th:pgsd_ho}
Let $x \in \real^d$, $t>0$ and $C>0$. Then:
\begin{gather*}
    U_t \I_{\real^d}(x) \leq C e^{-\lambda_0 t} \phi_0(x)
    \quad\iff\quad
    |x| \leq \sqrt{\left(e^{4t}+1\right)\left(\log C - \frac{d}{2}\log (2 \sqrt{\pi}) + \frac{d}{2} \log (1+e^{-4t})\right)}.
\end{gather*}
In particular, if $C < (2\sqrt{\pi})^{d/2}$, the right-hand side is void for large $t$, and if $C = (2\sqrt{\pi})^{d/2}$, then the inequality on the right hand side defines a bounded set of $x$ for large $t$'s.
\end{proposition}

\begin{proof}
Because of \eqref{eq:HO_gs} and \eqref{eq:HO_sem} the inequality on the left hand side reads
\begin{gather*}
    \exp\left[{|x|^2 \left(\frac{e^{-2t}}{e^{2t}+e^{-2t}}\right)}\right]
    \leq C \pi^{-\frac{d}{4}} \left(\frac{e^{2t}+e^{-2t}}{2}\right)^{\frac{d}{2}} e^{-dt}
    = C (2 \sqrt{\pi})^{-\frac{d}{2}} (1+e^{-4t})^{\frac{d}{2}},
\end{gather*}
which can be equivalently rewritten as
\begin{gather*}
    |x|^2 \leq \big(1+e^{4t}\big) \left(\log C - \frac{d}{2}\log (2\sqrt{\pi}) + \frac{d}{2} \log (1+e^{-4t})\right).
\qedhere
\end{gather*}
\end{proof}

Assume now that $C > (2\sqrt{\pi})^{d/2}$ and write
\begin{gather*}
    \rho(t):= K e^{2t} \sqrt{\left(1+e^{-4t}\right)\left(1 + \frac{d}{2K^2} \log (1+e^{-4t})\right)}, \quad t \geq 1,
\end{gather*}
where $K = \sqrt{\log\big(C/(2\sqrt{\pi})^{d/2}\big)}$. Clearly, we have $\rho(t) = K e^{2t} (1+o(1))$, as $t \to \infty$. It follows from Proposition~\ref{th:pgsd_ho} that the semigroup $\left\{U_t: t \geq 0 \right\}$ is pGSD with the exhausting family $\left\{K_t: t \geq 1\right\}$ where $K_t = \overline{B}_{\rho(t)}(0)$, i.e.
\begin{gather} \label{eq:pGSD_HO}
    U_t \I_{\real^d}(x) \leq C e^{-\lambda_0 t} \phi_0(x), \quad \text{for} \quad |x| \leq \rho(t), \ \ t \geq 1.
 \end{gather}
Therefore, we can apply Theorem~\ref{th:pGSD-to-puqe} to show that the progressive exponential uniform quasi-ergodicity holds on any space $L^p$; it is easily seen from \eqref {eq:HO_sem} that the assumption \eqref{A4} holds for every $t_0>0$. More precisely, for every $a \in (0,1)$ and $p \in [1,\infty]$ there are constants $C_1, C_2 >0$ such that
\begin{gather} \label{eq:peqe_HO}
    \sup_{\sigma \in \cM^1\big(\overline{B}_{\rho(at)}(0)\big)} \sup_{\substack{f \in  L^p(M,\mu)\\ \left\|f\right\|_p \leq 1}} \left|\frac{\sigma(U_t f)}{\sigma(U_t \I_{\real^d})} - m(f)\right|
    \leq C_1 e^{-C_2 t},
\end{gather}
for sufficiently large $t>1$.

It is interesting to note that one can use \eqref{eq:HO_kernel} and \eqref{eq:HO_sem} to verify the following: for a family of measures $\big\{\sigma_t\big\}_{t \geq 1}$ given by $\sigma_t = \delta_{x_t}$ with $x_t = (e^{2t}x_0)/2$ where $x_0 \in \R^d$ is fixed and for $f = \I_K$ where $K$ is a Borel set of finite Lebesgue measure we have
\begin{gather} \label{eq:nonuniform_HO}
    \frac{U_t f(x_t)}{U_t \I_{\real^d}(x_t)}
    =  \frac{\sigma_t(U_t f)}{\sigma_t(U_t \I_{\real^d})}
    \xrightarrow{t\to\infty} \frac{1}{(2\pi)^{d/2}}\int_K e^{\frac{(y-x_0)^2}{2}}\,dy.
	\end{gather}
This was first noticed in \cite[p.\ 135]{bib:KP} for $d=1$; the aim was to show that the quasi-ergodicity property for the harmonic oscillator is not uniform in $\real$ (note that the authors of that paper work with a Schr\"odinger operator with rescaled kinetic term, i.e.\ $H f = (1/2) f'' + |x|^2 f$).
		
We can now see more. Indeed, \eqref{eq:nonuniform_HO} shows that the convergence in \eqref{eq:peqe_HO} cannot be uniform in the sets $K_t$ that come from the pGSD property \eqref{eq:pGSD_HO} -- one can choose $x_0$ is such a way that $x_t \in K_t$ for sufficiently large $t$. This means that Theorem~\ref{th:pGSD-to-puqe} is optimal in the sense that, in general, the pGSD property with the exhausting family $\left\{K_{t} : t \geq t_0\right\}$ implies the progressive exponential quasi-ergodicity with $\left\{K_{at} : t \geq t_0/\theta\right\}$ only for $a < 1$, cf.\ Remark~\ref{rem:optimal_theta}.

\subsection{Semigroups of non-local Schr\"odinger operators in the DJP setting} \label{subsec-schroedinger}
One of the main motivations for this paper was to describe the quasi-ergodic properties of the Schr\"odinger semigroups studied in a recent paper \cite{bib:KSchi20}. These are evolution semigroups that belong to non-local Schr\"odinger operators of the form $H=-L+V$, where $L$ is the generator of a symmetric L\'evy process $\{X_t: t \geq 0\}$ (so-called \emph{L\'evy operator}) with the \emph{direct jump property} (DJP in short) in $\R^d$, $d \geq 1$, and $V$ is a sufficiently regular confining potential. Working in that generality, we identified in \cite{bib:KSchi20} a new regularity property of compact semigroups in $L^2$, that we called \emph{progressive intrinsic ultracontractivity} (pIUC in short). It means that the regularity of the semigroup improves as $t \to \infty$. We now want to understand what is the impact of pIUC on the quasi-ergodic properties of the semigroup.

The class of processes that we consider in this section consists of symmetric (self-dual) and strong Feller L\'evy processes (cf.\ Section \ref{sec:Levy}, also for notation). In addition, we have to impose some extra regularity conditions on the density $p_t(x)$ and the L\'evy measure $\nu$ (Assumptions (A1)--(A2) in \cite{bib:KSchi20}; we do not give details here). The key property is the following: we assume that the L\'evy measure $\nu(dx)$ has a density $\nu(x) \, dx$ and that there exists a decreasing \enquote{profile} function $f:(0,\infty) \to (0,\infty)$ such that $\nu(x) \asymp f(|x|)$, $x \neq 0$, and
\begin{align} \label{eq:DJP}
f_1 * f_1(x) \leq c f_1(x), \quad x \in \R^d,
\end{align}
for a constant $c >0$, where $f_1:= f \wedge 1$. This condition has a very suggestive probabilistic interpretation and, therefore, it is called the \emph{direct jump property} (DJP). Note that this class contains many examples of jump L\'evy processes and the corresponding L\'evy operators which play an important role in various applications. For example, it includes the
\begin{itemize}
\item \emph{fractional Laplace operator} $L=-(-\Delta)^{\alpha/2}$, $\alpha \in (0,2)$ (being the generator of the \emph{isotropic $\alpha$-stable process}); here $\nu(x) = c_{d,\alpha} |x|^{-d-\alpha}$;
\item \emph{relativistic operator} $L=-(-\Delta+m^{2/\alpha})^{\alpha/2}+m$, $\alpha \in (0,2)$, $m>0$ (being the generator of the \emph{isotropic relativistic $\alpha$-stable process}); here $\nu(x) \asymp e^{-m^{1/\alpha}|x|} (1 \wedge |x|)^{-\frac{d+\alpha-1}{2}} |x|^{-\frac{d+\alpha+1}{2}}$.
\end{itemize}

We assume that $V \in L^{\infty}_{\loc}(\R^d,dx)$ is a confining potential (i.e.\ $V(x) \to \infty$ as $|x| \to \infty$) such that there is a an increasing profile $g:[0,\infty) \to (0,\infty)$, growing at infinity not faster than an exponential function, and there are constants $c_1, c_2 \geq 1$ such that $c_1^{-1} g(|x|) \leq V(x) \leq c_1 g(|x|)$, for all $|x| \geq c_2$, see the detailed statement in \cite[Assumption (A3)]{bib:KSchi20}.

The Schr\"odinger operator $H=-L+V$ is well-defined, bounded below, and self-adjoint in $L^2(\R^d,dx)$. We denote by $U_t = e^{-tH}$, $t \geq 0$, the operators of the corresponding evolution semigroup. They are known to have the following probabilistic Feynman--Kac representation
\begin{align} \label{eq:djp_F-k}
    U_tf(x) = \ex^x \left[e^{-\int_0^t V(X_s)\,ds} f(X_t)\right], \quad f \in L^2(\R^d,dx), \ t>0.
\end{align}
The operators $U_t$ are self-adjoint, and by \cite[Lemma 2.3]{bib:KSchi20} (or Lemma \ref{lem:general-Feller-ass}) and Lemma~\ref{lem:general-heat-content}.\ref{lem:general-heat-content-a}, all regularity assumptions \eqref{A0}--\eqref{A3} are satisfied and \eqref{A4} follows from the proof of Lemma~\ref{lem:general-heat-content}.\ref{lem:general-heat-content-a}.
Our standard reference for self-adjoint Schr\"{o}dinger operators and the Feynman--Kac semigroups is the monograph \cite{bib:DC}.

We will now show how our present results apply to these semigroups. The discussion will be divided into three parts.

\medskip\noindent
\textbf{(1)} \textit{Heat content and exponential quasi-ergodicity.} First of all, due to Lemma~\ref{lem:general-heat-content}.\ref{lem:general-heat-content-b}, we know that
\begin{align} \label{eq:dsp_sc_hc}
    \liminf_{|x| \to \infty} \frac{V(x)}{\log|x|} >0 \qquad \Longleftrightarrow \qquad \exists \, t_1>0 \::\: Z(t_1)< \infty,
\end{align}
and our results in Section~\ref{sec:content} apply. More precisely, by Theorem~\ref{th:hcf_and_eqe} the semigroup $\left\{U_t: t \geq 0 \right\}$ is exponentially quasi-ergodic for every finite initial distribution, and the limiting measure
\begin{gather*}
    m(dx) =  \frac{\phi_0(x)}{\left\|\phi_0\right\|_1}\, dx,
\end{gather*}
where $\phi_0(x)$ is the ground state of the operator $H$, is the only quasi-stationary measure of $\left\{U_t: t \geq 0 \right\}$. Furthermore, in Corollary~\ref{cor:hcf_uqe} we obtain a pointwise version with the space-rate $U_{t_1}\I_{\real^d}(x)/\phi_0(x)$ which is uniform on compact sets. By \cite[Corollary 2.2]{bib:KL15a} there is a constant $c_3>0$ such that
\begin{gather*}
    \phi_0(x) \geq c_6 \left(1 \wedge \frac{\nu(x)}{V(x)}\right), \quad x \in \R^d
\end{gather*}
(this remains true for potentials without radial profiles). On the other hand, we always have
\begin{gather*}
    U_{t_1}\I_{\real^d}(x) \leq e^{\left\|V_-\right\|_{\infty} t_1}, \quad x \in \R^d,
\end{gather*}
which leads to the following explicit estimate of the rate
\begin{align} \label{eq:hc_rate}
    \frac{U_{t_1}\I_{\real^d}(x)}{\phi_0(x)} \leq c_3 \frac{V(x)}{\nu(x)}, \quad x \in \R^d.
\end{align}
The function on the right hand side tends to infinity as $|x| \to \infty$, but it is bounded on compact sets.
Better estimates for $U_{t_1}\I_{\real^d}(x)$ can be found in \cite[Theorem 4.8]{bib:KSchi20}.
Observe that the sufficient condition in \eqref{eq:dsp_sc_hc} does not depend on $\nu$.

\medskip\noindent
\textbf{(2)} \textit{aGSD regime and uniform exponential quasi-ergodicity.} By \cite[Theorem 2.6]{bib:KL15a},
\begin{align} \label{eq:dsp_sc_agsd}
    \liminf_{|x| \to \infty} \frac{V(x)}{|\log\nu(x)|} >0
    \iff
    \exists \, t_1, c_4>0 \ \forall \, x \in \R^d \::\: U_{t_1}\I_{\real^d}(x) \leq c_4 \phi_0(x).
\end{align}
Recall that the property on the right hand side is called aGSD, see Definition \ref{def:pGSD}.\ref{def:pGSD-a}. If we are in the aGSD regime -- that is, if the growth of the potential is sufficiently large with respect to $|\log\nu(x)|$ --, then the exponential quasi-ergodicity in Corollary~\ref{cor:hcf_uqe} discussed above becomes uniform on the full space $\R^d$. This is exactly what we saw in Corollary \ref{cor:agsd-euqe}.\ref{cor:agsd-euqe-a}. The same conclusion follows from \cite[Theorem 1, Corollary 2]{bib:KP}. However, now we get much more. Combining Corollary \ref{cor:agsd-euqe} and \eqref{eq:dsp_sc_agsd} we get that the uniform quasi-ergodicity of the semigroup $\left\{U_t: t \geq 0 \right\}$ is equivalent to the condition
\begin{gather*}
    \liminf_{|x| \to \infty} \frac{V(x)}{|\log\nu(x)|} >0.
\end{gather*}
If $\nu(x)$ has polynomial decay at infinity (as is, e.g., the case for the fractional Laplacian), then $|\log \nu(x)| \asymp \log|x|$, for $|x|$ large enough, and the properties in \eqref{eq:dsp_sc_hc} and \eqref{eq:dsp_sc_agsd} coincide. On the other hand, if $\nu$ decays at infinity faster than any polynomial, then \eqref{eq:dsp_sc_agsd} is more restrictive than \eqref{eq:dsp_sc_hc}. For example, if $\nu(x)$ is exponential at infinity (e.g.\ for relativistic operators), then $|\log \nu(x)| \asymp |x|$, for large $|x|$.

\medskip\noindent
\textbf{(3)} \textit{pGSD and progressive uniform quasi-ergodicity.} Finally, we characterize the quasi-ergodic behaviour of the Schr\"odinger semigroups $\left\{U_t: t \geq 0 \right\}$ with confining potentials as above which do not satisfy \eqref{eq:dsp_sc_agsd}, no matter how slow the potential may grow at infinity; this brings us out of the aGSD regime. Note that this problem was completely open. Since we want to be here as sharp as possible, we have to monitor precisely the threshold between the aGSD and the non-aGSD regime. For this reason, we follow \cite[Assumption (A4)]{bib:KSchi20} and impose an additional technical condition on the profile $g$ of the potential $V$. More precisely, we require that $g$ depends on the profile $f$ of the L\'evy density $\nu$ in a sufficiently regular way:\ we consider $R_0>0$ such that $f(R_0)<1$ and assume that $g(r) = h(|\log f(r)|)$, $r \geq R_0$, for an increasing function $h:[|\log f(R_0)|,\infty) \to (0,\infty)$ such that $h(s)/s$ is monotone. Under this assumption, the class of Schr\"odinger semigroups $\left\{U_t: t \geq 0 \right\}$ with L\'evy densities $\nu$ and potentials $V$ is divided into two classes:
\begin{enumerate}
\item \textbf{(aGSD regime):} $\liminf_{|x| \to \infty} \frac{V(x)}{|\log\nu(x)|} >0$, i.e.\ \eqref{eq:dsp_sc_agsd} holds;
\item \textbf{(non-aGSD regime):} $\lim_{|x| \to \infty} \frac{V(x)}{|\log\nu(x)|} =0$,
\end{enumerate}
cf.\ \cite[Remark 5.1]{bib:KSchi20}. We already know from Part~(2) that in the aGSD regime the Schr\"odinger semigroup is uniformly exponentially quasi-ergodic and, in fact, we have the equivalence of these two properties.

The rest of this section will be devoted to the analysis of the quasi-ergodic behaviour in the non-aGSD regime. Of course, if \eqref{eq:dsp_sc_hc} holds, then we know from Part~(1) above that we have the exponential quasi-ergodicity with semi-explicit rate, which is uniform on compact sets. With \cite[Corollary 5.6 b)]{bib:KSchi20} we can now give even better estimates for the rate \eqref{eq:hc_rate}
\begin{gather*}
    \frac{U_{t_1}\I_{\real^d}(x)}{\phi_0(x)} \leq c_5 \left(1 \vee \frac{e^{-c_5 V(x)}}{\nu(x)} \right), \quad x \in \R^d,
\end{gather*}
for some $c_4, c_5>0$. Our theorems from Section \ref{sec:pGSD} now yield much stronger results. In the paper  \cite{bib:KSchi20} we identify a new large-time regularity property of compact semigroups, which we call progressive intrinsic ultracontractivity (pIUC). This property holds for semigroups considered here. In \cite[Corollary 5.6 b)]{bib:KSchi20} we show that pIUC implies the following two-sided sharp estimates:\ there exist $\theta \in (0,1)$, $t_2 >0$ and an increasing function $\rho : [t_2, \infty) \to (0,\infty)$, with $\lim_{t \to \infty} \rho(t) = \infty$, such that
\begin{align} \label{eq:pGSD_schr}
    U_{t}\I_{\real^d}(x) \asymp e^{-\lambda_0 t} \phi_0(x) \asymp e^{-\lambda_0 t} \left(1 \wedge \frac{\nu(x)}{V(x)}\right),
    \quad |x| \leq \rho(\theta t), \ t \geq t_2/\theta
\end{align}
(note that the comparison constants are uniform in $x$ and $t$!). The function $\rho$ is the right-continuous generalized inverse function of $r \mapsto |\log f(r)|/g(r)$, see \cite[Lemma 5.2]{bib:KSchi20}. Note that \eqref{eq:pGSD_schr} is the \emph{progressive ground state domination} (pGSD) property, cf.\ Definition \ref{def:pGSD}.\ref{def:pGSD-b}.

Therefore, we see that $\left\{U_t: t \geq 0 \right\}$ is pGSD with the exhausting family $\left\{K_t: t \geq t_2/\theta \right\}$ where $K_t = \overline{B}_{\rho(\theta t)}(0)$, and we can apply Theorem~\ref{th:pGSD-to-puqe} to show that the progressive uniform quasi-ergodicity on any space $L^p$ holds. We obtain that there exist $\gamma, t_0>0$ such that for every $a, b \in (0,1)$ such that $a+2b=1$ and $p \in [1,\infty]$ there is a constant $c_6>0$ such that
\begin{gather*} 
    \sup_{\sigma \in \cM^1\big(\overline{B}_{\rho(a \theta t)}(0)\big)} \sup_{\substack{f \in  L^p(M,\mu)\\ \left\|f\right\|_p \leq 1}} \left|\frac{\sigma(U_t f)}{\sigma(U_t \I_{\real^d})} - m(f)\right|
    \leq c_6 \kappa_b(t), \quad t>(4/(a \wedge b)) t_0,
\end{gather*}
where
\begin{gather*}
\kappa_b(t):=e^{-\gamma b t} + \sup_{|x| \geq \rho(b \theta t)} U_{t_0}\I_{\real^d}(x).
\end{gather*}
Because of Theorem \ref{th:pGSD-gives-uniq}, the measure $m$ is the only quasi-stationary distribution of the Schr\"odinger
semigroup $\left\{U_t: t \geq 0 \right\}$ -- no matter how slowly $V$ grows at infinity. We need to verify only the condition \eqref{eq:add_ass_for_uniq}. To this end, we observe that by \cite[Corollary 5.6 b)]{bib:KSchi20} there are constants $c_7, c_8 >0$ such that
\begin{gather*}
   e^{\lambda_0 t} U_t \I_{\real^d}(x)
    \leq c_7 e^{-(c_8 V(x) - \lambda_0) t},
    \quad |x| \geq \rho(\theta t),
\end{gather*}
for sufficiently large values of $t$. Since $V(x) \to \infty$ as $|x| \to \infty$ and $\rho(t) \to \infty$ as $t \to \infty$, this implies \eqref{eq:equiv_to_add_ass_for_uniq}, which is equivalent to \eqref{eq:add_ass_for_uniq}.

\begin{example}
Let us illustrate these results for two specific classes of L\'evy measures and potentials. See Section~\ref{sec:Levy} for the notation.
\begin{enumerate}
\item \textbf{(polynomial L\'evy densities}, cf.\ \cite[Section 5.4, Example 5.7] {bib:KSchi20}\textbf{)}: let
\begin{gather*}
\nu(x) \asymp |x|^{-d-\alpha} (e \vee |x|)^{-\delta}, \quad \alpha \in (0,2), \ \delta \geq 0,
\end{gather*}
and let $V(x) = (1 \vee \log|x|)^{\beta}$, $\beta >0$. Then we have the following:
\begin{itemize}
\item if $\beta \geq 1$, then the corresponding Schr\"odinger semigroup is in the aGSD regime and exponential uniform quasi-ergodicity holds, see Part~(2) above;
\item if $\beta \in (0,1)$, then the semigroup is in the non-aGSD regime; it is pGSD and progressive uniform quasi-ergodicity holds, see Part~(3). We have
\begin{align}\label{eq:rates_ex}
\rho(t) = \exp\left(t^{1/(1-\beta)} \right) \quad \text{and} \quad \kappa_b(t) \asymp \exp\left(-c_9 \left(t \wedge t^{\frac{\beta}{1-\beta}}\right)\right),
\end{align}
for large values of $t$, with different constants $c_9$ in the lower and in the upper bound. The rate $\kappa_b(t)$ is exponential for $\beta \geq 1/2$, and subexponential (stretched-exponential) for $\beta < 1/2$.
\end{itemize}

\item \textbf{(exponential L\'evy densities}, cf.\ \cite[Section 5.5, Example 5.10]{bib:KSchi20}\textbf{)}: let
\begin{gather*}
\nu(x) \asymp e^{-m|x|} (1 \wedge |x|)^{-d-\alpha} (1 \vee |x|)^{-\delta}, \quad m>0, \ \alpha \in (0,2), \ \delta > (d+1)/2,
\end{gather*}
and let $V(x) = (1 \vee |x|)^{\beta}$, $\beta >0$. Then we have the following:
\begin{itemize}
\item if $\beta \geq 1$, then the corresponding Schr\"odinger semigroup is in the aGSD regime and exponential uniform quasi-ergodicity holds, see Part~(2) above;
\item if $\beta \in (0,1)$, then the semigroup is in the non-aGSD regime; it is pGSD and progressive uniform quasi-ergodicity holds, see Part~(3). We have
\begin{gather*}
\rho(t) = t^{1/(1-\beta)}
\end{gather*}
for large values of $t$, but the rate $\kappa_b(t)$ takes the same form as that in \eqref{eq:rates_ex}.
\end{itemize}
\end{enumerate}
\end{example}
Theorem \ref{th:pGSD} also shows that any semigroup $\left\{U_t: t \geq 0 \right\}$ with finite heat content (for large times) is automatically pGSD; moreover progressive uniform quasi-ergodicity holds with exponential time-rate. For polynomial L\'evy densities this result is not interesting as we already know that in this case \eqref{eq:dsp_sc_hc} is equivalent to \eqref{eq:dsp_sc_agsd}. However, for faster decays the assertions of the theorem are non-trivial. For instance, in the exponential case, one can check that Theorem \ref{th:pGSD} implies pGSD and progressive uniform quasi-ergodicity with the exhausting family $\left\{K_t: t \geq t_0 \right\}$ where $K_t = \overline{B}_{c t}(0)$, for a constant $c>0$.

\end{document}